\documentclass{amsart}

\usepackage{graphicx}
\usepackage{amssymb}
\usepackage{amscd}
\usepackage{latexsym}
\usepackage{amsfonts}
\usepackage{ragged2e}
\usepackage{amsmath}
\usepackage{float}
\usepackage[latin1]{inputenc}
\usepackage[english]{babel}

\newtheorem{proposition}{Proposition}
\newtheorem{lemma}{Lemma}
\newtheorem{theorem}{Theorem}
\newtheorem{corollary}{Corollary}
\newtheorem{definition}{Definition}
\newtheorem{remark}{Remark}

\newcommand\zkn[1][k]{Z_{#1}^{(n)}}

\DeclareMathOperator{\SU}{SU} \DeclareMathOperator\sll{sl}
\newcommand{\im}{\mathop{\fam0 Im}\nolimits}

\newcommand{\tr}{\mathop{\fam0 Tr}\nolimits}
\newcommand{\re}{\mathop{\fam0 Re}\nolimits}

\newcommand{\Hom}{\mathop{\fam0 Hom}\nolimits}

\newcommand{\End}{\mathop{\fam0 End}\nolimits}

\newcommand{\id}{\mathop{\fam0 Id}\nolimits}

\newcommand{\Aut}{\mathop{\fam0 Aut}\nolimits}

\newcommand{\Id}{\mathop{\fam0 Id}\nolimits}

\newcommand{\bC}{{\mathbb C}}
\newcommand{\bR}{{\mathbb R}}
\newcommand{\bT}{{\bar T}}
\newcommand{\C}{C}

\newcommand{\Z}{{\mathbb Z}}
\newcommand{\bZ}{\Z{}}

\newcommand{\D}{{\mathcal D}}
\newcommand{\ra}{\mathop{\fam0 \rightarrow}\nolimits}

\newcommand{\cH}{ {\mathcal H}}

\newcommand{\T}{ {\mathcal T}}

\newcommand{\V}{ {\mathcal V}}
\renewcommand{\L}{{\mathcal L}}

\renewcommand{\P}{ {\mathbb P}}

\newcommand{\tD}{ {\tilde D}}

\newcommand{\s}{\sigma}

\newcommand{\Nabla}{{\mathbf {\hat \nabla}}}
\newcommand{\Nablat}{{\mathbf {\hat \nabla}}^{t}}
\newcommand{\Nablae}{{\mathbf {\hat \nabla}}^e}
\newcommand{\Nablaet}{{\mathbf {\hat \nabla}}^{e,t}}
\newcommand{\BTstar}{\star^{\text{\tiny BT}}}
\newcommand{\tBTstar}{{\tilde \star}^{\text{\tiny BT}}}

\newcommand{\Teim}{Teichm{\"u}ller }
\newcommand{\Ric}{\mathop{\fam0 Ric}\nolimits}

\usepackage[draft]{fixme}
\usepackage[all,line]{xy}
\usepackage{enumitem}

\newcommand{\cL}{{\mathcal L}}
\newcommand{\Hk}{H^{(k)}}

\newcommand{\pik}{\pi^{(k)}}
\newcommand{\bH}{\mathbb{H}}
\renewcommand{\C}{\mathbb{C}}

\providecommand{\abs}[1]{\left|#1\right|} 
\providecommand{\norm}[1]{\lVert#1\rVert}
\providecommand{\inner}[2]{\left \langle #1,#2 \right \rangle}

\providecommand{\formal}[1]{[[ #1 ]]}
\providecommand{\derx}[1]{\frac{\partial}{\partial x_{#1}}}
\providecommand{\dery}[1]{\frac{\partial}{\partial y_{#1}}}
\providecommand{\derz}[1]{\frac{\partial}{\partial z_{#1}}}
\providecommand{\derbz}[1]{\frac{\partial}{\partial \bar{z}_{#1}}}
\providecommand{\dert}[1]{\frac{\partial}{\partial Z_{#1}}}
\providecommand{\derbt}[1]{\frac{\partial}{\partial \bar{Z}_{#1}}}

\DeclareMathOperator{\U}{U}

\begin{document}

\title{Asymptotics of Toeplitz operators
 and applications in TQFT}

\author{J{\o}rgen Ellegaard Andersen}\address{Centre for Quantum Geometry
  of Moduli Spaces\\ Faculty of Science and Technology\\
  Aarhus University\\
  DK-8000, Denmark}

\email{andersen@imf.au.dk}

\author{Jakob Lindblad Blaavand} \address{Centre for Quantum Geometry
  of Moduli Spaces\\ Faculty of Science and Technology\\
  Aarhus University\\
  DK-8000, Denmark}

\email{blaavand@imf.au.dk}

\maketitle

\begin{abstract}
  In this paper we provide a review of asymptotic results of Toeplitz
  operators and their applications in TQFT. To do this we review the
  differential geometric construction of the Hitchin connection on a
  prequantizable compact symplectic manifold. We use asymptotic
  results relating the Hitchin
  connection and Toeplitz operators, to, in the special case of the moduli space of flat
  $\SU(n)$-connections on a surface, prove asymptotic faithfulness of
  the $\SU(n)$ quantum representations of the mapping class group. We
  then go on to review formal Hitchin connections and formal
  trivializations of these. We discuss how these fit together to
  produce a Berezin--Toeplitz star product, which is independent of
  the complex structure. Finally we give explicit examples of all the above objects in
  the case of the abelian moduli space. We furthermore discuss an
  approach to curve operators in the TQFT associated to abelian
  Chern--Simons theory.
\end{abstract}

\section{Introduction}

Witten constructed, via path integral techniques, a quantization of
Chern-Simons theory in $2+1$ dimensions, and he argued in \cite{W1}
that this produced a TQFT, indexed by a compact simple Lie group and
an integer level $k$. For the group $\SU(n)$ and level $k$, let us
denote this TQFT by $\zkn$. Combinatorially, this theory was first constructed by Reshetikhin and
Turaev, using representation theory of $U_q(\sll(n,\bC))$ at
$q=e^{(2\pi i)/(k+n)}$, in \cite{RT1} and \cite{RT2}. Subsequently,
the TQFT's $\zkn$ were constructed using skein theory by Blanchet,
Habegger, Masbaum and Vogel in \cite{BHMV1}, \cite{BHMV2} and
\cite{B1}.

The two-dimensional part of the TQFT $\zkn$ is a modular functor with
a certain label set. For this TQFT, the label set $\Lambda_k^{(n)}$ is
a finite subset (depending on $k$) of the set of finite dimensional
irreducible representations of $\SU(n)$. We use the usual labeling of
irreducible representations by Young diagrams, so in particular
$\Box\in \Lambda_k^{(n)}$ is the defining representation of
$\SU(n)$. Let further $\lambda_0^{(d)} \in \Lambda_k^{(n)}$ be the
Young diagram consisting of $d$ columns of length $k$. The label set
is also equipped with an involution, which is simply induced by taking
the dual representation. The trivial representation is a special
element in the label set which is clearly preserved by the involution.

\begin{align*}
  \zkn: \quad \left\{\parbox{3cm}{\RaggedRight Category of (extended)
      closed oriented surfaces with $\Lambda^{(n)}_k$-labeled marked
      points with projective tangent vectors} \right\} \ \to \
  \left\{\parbox{3cm}{\RaggedRight Category of finite dimensional
      vector spaces over $\bC$} \right\}
\end{align*}

The three-dimensional part of $\zkn$ is an association of a vector,
$$\zkn(M,L,\lambda)\in \zkn(\partial M, \partial L, \partial \lambda),$$
to any compact, oriented, framed $3$--manifold $M$ together with an
oriented, framed link $(L,\partial L)\subseteq (M,\partial M)$ and a
$\Lambda_k^{(n)}$-labeling $\lambda : \pi_0(L) \ra \Lambda_k^{(n)}$.


This association has to satisfy the Atiyah-Segal-Witten TQFT axioms
(see e.g. \cite{At}, \cite{Segal} and \cite{W1}). For a more
comprehensive presentation of the axioms, see Turaev's book \cite{T}.

The geometric construction of these TQFTs was proposed by Witten in
\cite{W1} where he derived, via the Hamiltonian approach to quantum
Chern-Simons theory, that the geometric quantization of the moduli
spaces of flat connections should give the two-dimensional part of the
theory. Further, he proposed an alternative construction of the
two-dimensional part of the theory via WZW-conformal field
theory. This theory has been studied intensively. In particular, the
work of Tsuchiya, Ueno and Yamada in \cite{TUY} provided the major
geometric constructions and results needed. In \cite{BK}, their
results were used to show that the category of integrable highest
weight modules of level $k$ for the affine Lie algebra associated to
any simple Lie algebra is a modular tensor category. Further, in
\cite{BK}, this result is combined with the work of Kazhdan and
Lusztig \cite{KL} and the work of Finkelberg \cite{Fi} to argue that
this category is isomorphic to the modular tensor category associated
to the corresponding quantum group, from which Reshetikhin and Turaev
constructed their TQFT. Unfortunately, these results do not allow one
to conclude the validity of the geometric constructions of the
two-dimensional part of the TQFT proposed by Witten.  However, in
joint work with Ueno, \cite{AU1}, \cite{AU2}, \cite{AU3} and
\cite{AU4}, the first author have given a proof, based mainly on the results of
\cite{TUY}, that the TUY-construction of the WZW-conformal field
theory, after twist by a fractional power of an abelian theory,
satisfies all the axioms of a modular functor.  Furthermore, we have
proved that the full $2+1$-dimensional TQFT resulting from this is
isomorphic to the aforementioned one, constructed by BHMV via skein
theory. Combining this with the theorem of Laszlo \cite{La1}, which
identifies (projectively) the representations of the mapping class
groups obtained from the geometric quantization of the moduli space of
flat connections with the ones obtained from the TUY-constructions,
one gets a proof of the validity of the construction proposed by
Witten in \cite{W1}.

Another part of this TQFT is the quantum $\SU(n)$ representations of the
mapping class groups. Namely, if $\Sigma$ is a closed oriented
surfaces of genus $g$, $\Gamma$ is the mapping class group of
$\Sigma$, and $p$ is a point on $\Sigma$, then the modular functor
induces a representation
\begin{equation}\label{rep}
  Z^{(n,d)}_{k} : \Gamma \to \P\Aut\bigl(\zkn (\Sigma, p, \lambda_0^{(d)})\bigr).
\end{equation}
For a general label of $p$, we would need to choose a projective
tangent vector $v_p\in T_p\Sigma/\bR_+$, and we would get a
representation of the mapping class group of $(\Sigma,p, v_p)$.  But
for the special labels $\lambda_0^{(d)}$, the dependence on $v_p$ is
trivial and in fact we get a representation of $\Gamma$.


Let us now briefly recall the geometric construction of the
representations $Z^{(n,d)}_k$ of the mapping class group, as proposed
by Witten, using geometric quantization of moduli spaces.

We assume from now on that the genus of the closed oriented surface
$\Sigma$ is at least two. Let $M$ be the moduli space of flat $\SU(n)$
connections on $\Sigma - p$ with holonomy around $p$ equal to
$\exp(2\pi i d/n)\Id \in \SU(n)$. When $(n,d)$ are coprime, the moduli
space is smooth. In all cases, the smooth part of the moduli space has
a natural symplectic structure $\omega$.  There is a natural smooth
symplectic action of the mapping class group $\Gamma$ of $\Sigma$ on
$M$. Moreover, there is a unique prequantum line bundle $(\L,\nabla,
(\cdot,\cdot))$ over $(M,\omega)$. The Teichm\"{u}ller space $\T$ of
complex structures on $\Sigma$ naturally, and $\Gamma$-equivariantly,
parametrizes K\"{a}hler structures on $(M,\omega)$. For $\sigma\in \T
$, we denote by $M_\sigma$ the manifold $(M,\omega)$ with its
corresponding K\"{a}hler structure.  The complex structure on
$M_\sigma$ and the connection $\nabla$ in $\L$ induce the structure of
a holomorphic line bundle on $\L$. This holomorphic line bundle is
simply the determinant line bundle over the moduli space, and it is an
ample generator of the Picard group \cite{DN}.

By applying geometric quantization to the moduli space $M$, one gets,
for any positive integer $k$, a certain finite rank bundle over \Teim
space $\T$ which we will call the {\em Verlinde bundle} $\V^{(k)}$ at
level $k$. The fiber of this bundle over a point $\sigma\in \T$ is
$\V_{\sigma}^{(k)} = H^0(M_\sigma,\L^k)$. We observe that there is a
natural Hermitian structure $\langle\cdot,\cdot\rangle$ on
$H^0(M_\sigma,\L^k)$ by restricting the $L_2$-inner product on global
$L_2$ sections of $\L^k$ to $H^0(M_\sigma,\L^k)$.

The main result pertaining to this bundle is:
\begin{theorem}[Axelrod, Della Pietra and Witten;
  Hitchin]\label{projflat}
  The projectivization of the bundle $\V^{(k)}$ supports a natural flat
  $\Gamma$-invariant connection $\Nabla$.
\end{theorem}

This is a result proved independently by Axelrod, Della Pietra and
Witten \cite{ADW} and by Hitchin \cite{H}. In section \ref{ghc}, we
review our differential geometric construction of the connection
$\Nabla$ in the general setting discussed in \cite{A5}. We obtain as a
corollary that the connection constructed by Axelrod, Della Pietra and
Witten projectively agrees with Hitchin's.

Because of the existence of this connection, the $2$-dimensional part
of the modular functor $Z^{(n)}_k$ is the vector space $\P(V^{(k)})$
of covariant constant sections of $\P(\V^{(k)})$ over Teichm\"{u}ller
space $\T$.

\begin{definition}\label{MD1}
  We denote by $Z^{(n,d)}_{k}$ the representation,
  \begin{align*}
    Z^{(n,d)}_{k} : \Gamma \to \Aut\bigl(\P (V^{(k)})\bigr),
  \end{align*}
  obtained from the action of the mapping class group on the covariant
  constant sections of $\P(\V^{(k)})$ over $\T$.
\end{definition}

The projectively flat connection $\Nabla$ induces a {\em flat}
connection $\Nablae$ in $\End(\V^{(k)})$. This flat connection can be
used to show asymptotically flatness of the quantum representations
$Z^{(n,d)}_k$,

\begin{theorem}[Andersen \cite{A3}]\label{MainA3} Assume that $g \geq 2$, $n$ and $d$ are coprime or
  that $(n,d)=(2,0)$ when $g=2$.  Then, we have that
  \begin{equation*}
    \bigcap_{k=1}^\infty \ker(Z^{(n,d)}_k) =
    \begin{cases}
      \{1, H\} & g=2 \mbox{, }n=2 \mbox{ and } d=0 \\ \{1\}&
      \mbox{otherwise},
    \end{cases}
  \end{equation*}
  where $H$ is the hyperelliptic involution.
\end{theorem}
In Section~\ref{sec:asympt-faithf} we discuss the proof of this
Theorem, and how it relies on the asymptotics of \emph{Toeplitz operators}
$T^{(k)}_f$
associated a smooth function $f$ on $M$. For each $f \in C^\infty(M)$
and each point $\s \in \T$ we have the Toeplitz operator, 
\[
T^{(k)}_{f,\s} : H^0(M_\s,\cL^k_\s) \to H^0(M_\s,\cL^k_\s),
\]
which is given by
\[
T^{(k)}_{f,\s} s= \pik_\s (fs)
\] for all $s \in H^0(M_\s,\cL^k_\s)$. Here $\pik_\s$ is the
orthogonal projection onto $H^0(M_\s,\cL^k_\s)$ induced from the
$L_2$-inner product on $C^\infty(M,\cL^k)$. We get a smooth section
of $\End(\V^{(k)})$,
\[
T^{(k)}_f \in C^\infty(\T,\End(\V^{(k)})),
\] by letting $T^{(k)}_f (\s) = T^{(k)}_{f,\s}$. See Section~\ref{BZdq}
for a discussion of the Toeplitz operators and their connection to
deformation quantization. The sections $T^{(k)}_f$ of $\End{\V^{(k)}}$
over $\T$ are not covariant constant with respect to
$\Nablae$. However, they are asymptotically as $k$ goes to
infinity. This is made precise when we discuss the formal Hitchin
Connection below.

The existence of a connection as above is not a unique thing for the
moduli spaces, the construction can be generalized to a general
compact prequantizable symplectic manifold $(M, \omega)$ with prequantum
line bundle $(\cL,(\cdot,\cdot), \nabla)$. We assume that $\T$ is a
complex manifold which holomorphically and rigidly (see Definition
\ref{rigid}) parameterizes K\"{a}hler structures on
$(M,\omega)$. Then, the following theorem, proved in \cite{A5},
establishes the existence of the Hitchin connection (see
Definition~\ref{def:hitchin}) under a mild
cohomological condition.

\begin{theorem}[Andersen]\label{MainGHCI}
  Suppose that $I$ is a rigid family of K\"{a}hler structures on the
  compact, prequantizable symplectic manifold $(M,\omega)$ which
  satisfies that there exists an $n\in \bZ$ such that the first Chern
  class of $(M,\omega)$ is $n [\frac{\omega}{2\pi}]\in H^2(M,\bZ)$ and
  \mbox{$H^1(M,\bR) = 0$}. Then, the Hitchin connection $\Nabla$ in
  the trivial bundle $\cH^{(k)} = \mathcal{T} \times C^\infty(M,
  \mathcal{L}^k)$ preserves the subbundle $H^{(k)}$ with fibers
  $H^0(M_\sigma, \mathcal{L}^k)$. It is given by
  \[\Nabla_V = \Nablat_V + \frac1{4k+2n} \left\{\Delta_{G(V)} +
    2\nabla_{G(V)\cdot dF} + 4k V'[F]\right\},\] where $\Nablat$ is the
  trivial connection in $\cH^{(k)}$, and $V$ is any smooth vector
  field on $\T$. 
\end{theorem}
This result is discussed in much greater detail in Section~\ref{ghc},
where all ingredients are introduced.

In Section \ref{sec:form-hitch-conn}, we study the formal Hitchin
connection which was introduced in \cite{A5}. Let $\D(M)$ be the space
of smooth differential operators on $M$ acting on smooth functions on
$M$. Let $\C_h$ be the trivial $C^\infty_h(M)$-bundle over $\T$, where
$C^\infty_h(M)$ is formal power series with coefficients in $C^\infty(M)$.

\begin{definition}\label{fc2}
  A formal connection $D$ is a connection in $\C_h$ over $\T$ of the
  form
  \[D_V f = V[f] + \tD(V)(f),\] where $\tD$ is a smooth one-form on
  $\T$ with values in $\D_h(M) = \D(M)\formal{h}$, $f$ is any smooth
  section of $\C_h$, $V$ is any smooth vector field on $\T$ and $V[f]$
  is the derivative of $f$ in the direction of $V$.
\end{definition}

Thus, a formal connection is given by a formal series of differential
operators
\[\tD(V) = \sum_{l=0}^\infty \tD^{(l)}(V) h^l.\]

From Hitchin's connection in $H^{(k)}$, we get an induced connection
$\Nablae$ in the endomorphism bundle $\End(H^{(k)})$. As previously
mentioned, the Toeplitz operators are not covariant constant sections
with respect to $\Nablae$, but asymptotically in $k$ they are. This
follows from the properties of the formal Hitchin connection, which is
the formal connection $D$ defined through the following theorem
(proved in \cite{A5}).

\begin{theorem}(Andersen)\label{MainFGHCI2} There is a unique formal
  connection $D$ which satisfies that
  \begin{equation}
    \Nablae_V T^{(k)}_f \sim T^{(k)}_{(D_V f)(1/(k+n/2))}\label{Tdf2}
  \end{equation}
  for all smooth section $f$ of $\C_h$ and all smooth vector fields
  $V$ on $\T$. Moreover,
  \[\tD = 0 \mod h.\]
  Here $\sim$ means the following: For all $L\in \Z_+$ we have that
  \[\left\| \Nablae_V T^{(k)}_{f} - \left( T^{(k)}_{V[f]} +
      \sum_{l=1}^L T_{\tD^{(l)}_V f}^{(k)} \frac1{(k+n/2)^{l}}\right)
  \right\| = O(k^{-(L+1)}),\] uniformly over compact subsets of $\T$,
  for all smooth maps $f:\T \ra C^\infty(M)$.
\end{theorem}

Now fix an $f\in C^\infty(M)$, which does not depend on $\sigma \in
\mathcal{T}$, and notice how the fact that $\tilde D = 0 \mbox{ mod }
h$ implies that
\[\left\| \Nablae_V T^{(k)}_{f} \right\| = O(k^{-1}).\]
This expresses the fact that the Toeplitz operators are asymptotically
flat with respect to the Hitchin connection.

We define a mapping class group equivariant formal trivialization of
$D$ as follows.

\begin{definition}\label{formaltrivi2}
  A formal trivialization of a formal connection $D$ is a smooth map
  $P : \T \ra \D_h(M)$ which modulo $h$ is the identity, for all
  $\s\in \T$, and which satisfies
  \[D_V(P(f)) = 0,\] for all vector fields $V$ on $\T$ and all $f\in
  C^\infty_h(M)$. Such a formal trivialization is mapping class group
  equivariant if $P(\phi(\sigma)) = \phi^* P(\sigma)$ for all $\sigma
  \in \T$ and $\phi\in \Gamma$.
\end{definition}

Since the only mapping class group invariant functions on the moduli
space are the constant ones (see \cite{Go1}), we see that in the case
where $M$ is the moduli space, such a $P$, if it exists, must be
unique up to multiplication by a formal constant, i.e. an element of
$\bC_h = \bC\formal{h}$.

Clearly if $D$ is not flat, such a formal trivialization cannot exist
even locally on $\T$. However, if $D$ is flat and its zero-order term
is just given by the trivial connection in $C_h$, then a local formal
trivialization exists, as proved in \cite{A5}.

Furthermore, it is proved in \cite{A5} that flatness of the formal
Hitchin connection is implied by projective flatness of the Hitchin
connection. As was proved by Hitchin in \cite{H}, and stated above in
Theorem \ref{projflat}, this is the case when $M$ is the moduli
space. 

In Section~\ref{sec:form-hitch-conn} we discuss how this formal
trivialization of a formal connection give a way of defining a star
product from the Berezin--Topelitz star product, which turn out not to
depend on the complex structure $\s$. In
Section~\ref{sec:form-hitch-conn} we furthermore discuss the lower
order terms of formal trivialization and the star product.

In Section~\ref{sec:abelian} we consider the all of the above objects
in the case where the manifold $M$ is a principal polarized abelian
variety. We furthermore discuss abelian Chern--Simons theory and the
moduli space of $U(1)$-connections on a closed surface $\Sigma$. We
find a flat Hitchin connection on the $U(1)$-moduli space $M$ and find
a formal trivialization $P$ of the formal Hitchin connection. With
this formal trivialization we define the curve operators to a cylinder
$\Sigma\times [0,1]$ with a link $\gamma$ inside, to be the Toeplitz operator
associated to the corresponding holonomy function $h_\gamma$ on $M$,
$Z^{(k)} = T^{(k)}_{h_\gamma}$. With this definition of a curve
operator we show that
\[
\inner{h_{\gamma_1}}{h_{\gamma_2}} = \lim_{k\to \infty}\inner{Z^{(k)}(\Sigma,\gamma_1)}{Z^{(k)}(\Sigma,\gamma_2)}
\]
and as required by the TQFT axioms that
\[
Z^{(k)}(\Sigma\times S^1) = \dim(Z^{(k)}(\Sigma)).
\]

\section{The Hitchin connection}\label{ghc}

In this section, we review our construction of the Hitchin connection
using the global differential geometric setting of \cite{A5}. This
approach is close in spirit to Axelrod, Della Pietra and Witten's in
\cite{ADW}, however we do not use any infinite dimensional gauge
theory. In fact, the setting is more general than the gauge theory
setting in which Hitchin in \cite{H} constructed his original
connection. But when applied to the gauge theory situation, we get the
corollary that Hitchin's connection agrees with Axelrod, Della Pietra
and Witten's.

Hence, we start in the general setting and let $(M,\omega)$ be any
compact symplectic manifold.

\begin{definition}\label{prequantumb}
  A prequantum line bundle $(\L, (\cdot,\cdot), \nabla)$ over the
  symplectic manifold $(M,\omega)$ consist of a complex line bundle
  $\L$ with a Hermitian structure $(\cdot,\cdot)$ and a compatible
  connection $\nabla$ whose curvature is
  \begin{align*}
    F_\nabla(X,Y) = [\nabla_X, \nabla_Y] - \nabla_{[X,Y]} = -i \omega
    (X,Y).
  \end{align*}
  We say that the symplectic manifold $(M,\omega)$ is prequantizable
  if there exist a prequantum line bundle over it.
\end{definition}

Recall that the condition for the existence of a prequantum line
bundle is that $[\frac{\omega}{2\pi}]\in \im(H^2(M,\bZ) \ra
H^2(M,\bR))$. Furthermore, the inequivalent choices of prequantum line
bundles (if they exist) are parametriced by $H^1(M,U(1))$ (see
e.g. \cite{Woodhouse}).

We shall assume that $(M,\omega)$ is prequantizable and fix a
prequantum line bundle $(\L, (\cdot,\cdot), \nabla)$.


Before dwelving into the details we discuss general facts about families
of K\"{a}hler structures on a symplectic manifold.

\subsection*{Families of K\"{a}hler structures}
\label{sec:famil-kahl-struct}

From now on we assume $\T$ is a smooth manifold. Later we impose
extra structure.

  A \emph{family of K\"{a}hler structures} on a symplectic manifold
  $(M,\omega)$ parametrized by $\T$ is a map
\[
I : \T \to C^\infty(M,\End{TM}),
\]
that to each element $\sigma \in \T$ associates an integrable
and compatible almost complex structure. $I$ is said to be smooth if $I$
defines a smooth section of $\pi_M^* \End(TM) \to \T \times M$.

For each point $\sigma \in \T$ we define $M_\sigma$ to be $M$ with the
K\"{a}hler structure defined by $\omega$ and $I_\sigma := I(\sigma)$,
and the K\"{a}hler metric is denoted by $g_\sigma$.

Every $I_\sigma$ is an almost complex structure and hence induce a splitting of the complexified tangent bundle
$TM_\C$, denoted by $TM_\C = T_\sigma \oplus \bT_\sigma$, and the
projection to each factor is given by
\[
\pi^{1,0}_\sigma = \frac{1}{2}(Id - i I_\sigma) \quad \text{and} \quad
\pi^{0,1}_\sigma = \frac{1}{2}(Id + i I_\sigma).
\]

If $I_\sigma^2 =
-Id$ is differentiated
along a vector field $V$ on $\T$, we get
\[
V[I]_\sigma I_\sigma + I_\sigma V[I]_\sigma = 0,
\]
and hence $V[I]_\sigma$ changes types on $M_\sigma$. Then for each
$\sigma$, $V[I]_\sigma$ give an element of
\[
C^\infty(M,((\bT_\sigma)^* \otimes T_\sigma)\oplus ((T_\sigma)^*
\otimes \bT_\sigma)), 
\]
and we have a splitting $V[I]_\sigma = V[I]_\sigma' + V[I]_\sigma''$
where
\[
V[I]'_\sigma \in C^\infty(M,(\bT_\sigma)^* \otimes T_\sigma) \quad
\text{and} \quad V[I]''_\sigma \in C^\infty(M,(T_\sigma)^*\otimes \bT_\sigma).
\]
This splitting of $V[I]$ happens for every vector field on $\T$ and
actually induce an almost complex structure on $\T$.

Since $V[I]_\sigma$ is a smooth section of $TM_\C \otimes
T^*M_\C$ and the symplectic structure is a smooth section of
$T^*M_\C \otimes T^*M_\C$ we can define a bivector field
$\tilde{G}(V)$ by contraction with the symplectic form
\[
\tilde{G}(V) \cdot \omega = V[I].
\]
$\tilde{G}(V)$ is unique since $\omega$ is non-degenerate. By definition of the
K\"{a}hler metric, $g$ is the contraction of $\omega$ and $I$, $g =
\omega \cdot I$.  We use the
$\cdot$-notation for contraction in the following way
\[
g(X,Y) = (\omega \cdot I)(X,Y) = \omega(X,IY) \quad \text{and} \quad
g(X,Y) = -(I \cdot \omega)(X,Y) = -\omega(IX,Y).
\] Since $\omega$ is independent of $\sigma$ taking
the derivative of this identity in the direction of a vector field $V$
on $\T$ we obtain
\[
V[g] = \omega \cdot V[I] = \omega \cdot \tilde{G}(V) \cdot \omega.
\]
Since $g$ is symmetric so is $V[g]$, and with $\omega$ being
anti-symmetric $\tilde{G}(V)$ is symmetric. We furthermore have that
$\omega$ is of type $(1,1)$ when regarded as a K\"{a}hler form on
$M_\sigma$, using this and the fact that $V[I]_\sigma$ changes types
on $M_\sigma$ we get that $\tilde{G}(V)$ splits as $\tilde{G}(V) =
G(V) +\bar{G}(V)$, where
\[
G(V) \in C^\infty(M,S^2(T')) \quad \text{and} \quad \bar{G}(V) \in C^\infty(M,S^2(\bT)).
\]
In the above we have suppressed the dependence on $\sigma$, since this
is valid for any $\sigma$. 

\subsection*{Holomorphic families of K\"{a}hler structures}
\label{sec:holom-famil-kahl}

Let us now assume that $\T$ furthermore is a complex manifold. We can
then ask $I: \T \to C^\infty(M,\End(TM))$ to be holomorphic. By using
the splitting of $V[I]$ we make the following definition.
\begin{definition}
  Let $\T$ be a complex manifold and $I$ a smooth family of complex
  structures on $M$ parametrized by $\T$. Then $I$ is holomorphic if
\[
V'[I] = V[I]' \quad \text{and} \quad V''[I] = V[I]''
\]
for all vector fields $V$ on $\T$.
\end{definition}

Assume $J$ is an integrable almost complex structure on $\T$ induced
by the complex structure on $\T$. $J$ induces an almost complex
structure, $\hat{I}$ on $\T \times M$ by
\[
\hat{I}(V \oplus X) = JV \oplus I_\sigma X,
\]
where $V+X \in T_{(\sigma,p)(\T \times M)}$. In \cite{AGL} a simple
calculation shows that the Nijenhuis tensor on $\T \times M$ vanish
exactly when $\pi^{0,1}V'[I]X = 0$ and $\pi^{1,0}V''[I]X = 0$, which
by the Newlander--Nirenberg theorem shows that $\hat{I}$ is integrable
if and only if $I$ is holomorphic, hence the name. 

Remark that for a holomorphic family of K\"{a}hler structures on
$(M,\omega)$ we have 
\[
\tilde{G}(V') \cdot \omega = V'[I] = V[I]' = G(V)\cdot \omega,
\]
which implies $\tilde{G}(V') = G(V)$. We can in the same way show that
$\bar{G}(V) = \tilde{G}(V'')$.

\subsection*{Rigid families of K\"{a}hler structures}
\label{sec:rigid-famil-kahl}

In constructing an explicit formula for the Hitchin connection we need
the following rather restrictive assumption on our family of
K\"{a}hler structures.

\begin{definition} \label{rigid}
  A family of K\"{a}hler structures $I$ on $M$ is called \emph{rigid}
  if
\[
\nabla_{X''}G(V) = 0
\]
for all vector fields $V$ on $\T$ and $X$ on $M_\sigma$.
\end{definition}
Equivalently we could give the above equation in terms of the induced
$\bar{\partial}_\sigma$-operator on $M_\sigma$,
\[
\bar{\partial}_\sigma (G(V)_\sigma) = 0,
\]
for all $\sigma \in \T$ and all vector fields $V$ on $\T$.

There are several examples of rigid families of K\"{a}hler
structures, see e.g. \cite{AGL}. It should be remarked that this condition is also built into the
arguments of \cite{H}.

\subsection*{The Hitchin connection}
\label{sec:hitchins-connection}

Now all tools are defined and we can construct the Hitchin
connection. In Theorem~\ref{HCE} we need $M$ to be
compact, so let us assume this. Recall the quantum spaces
\[
H^{(k)}_\sigma = H^0(M_\s,\L^k_\s)\{s \in C^\infty(M,\cL^k) \, | \,
\nabla^{0,1}_\sigma s = 0\},
\]
where $\nabla^{0,1}_\sigma = \frac{1}{2}(Id +
iI_\sigma)\nabla$. 

It is not clear that these spaces form a vector bundle over $\T$. But by
constructing a bundle, where these sit as subspaces of each of the
fibers, and a connection in this bundle preserving $H^{(k)}_\sigma$,
$H^{(k)}$ will be a subbundle over $\T$.

Define the trivial bundle $\cH^{(k)} = \T \times C^\infty(M,\cL^k)$ of
infinite rank. The finite dimensional subspaces $H^{(k)}_\sigma$ sits
inside each of the fibers. This bundle has of course the trivial
connection $\nabla^t$, but we
seek a connection preserving $H^{(k)}_\sigma$.
\begin{definition} \label{def:hitchin}
  A \emph{Hitchin connection} is a connection $\hat{\nabla}$ in
  $\cH^{(k)}$, which preserves the subspaces $H^{(k)}_\sigma$, and is
  of the form
\[
\hat{\nabla} = \nabla^t + u,
\]
where $u \in \Omega^1(\T,\D(M,\cL^k))$ is a 1-form on $\T$ with values
in differential operators acting on sections of $\cL^k$.
\end{definition}

By analyzing the condition $\nabla^{0,1}_\sigma \hat{\nabla}_Vs = 0$ for
every vector field $V$ on $\T$, we hope to find an explicit expression for
$u$. If we express the above condition in terms of $u$, $u$ should
satisfy
\[
0 = \nabla^{0,1}_\sigma V[s] + \nabla^{0,1}_\sigma u(V)s,
\]
and if we differentiate $\nabla^{0,1}_\sigma s = 0$ along the a vector
field $V$ on $\T$ we get
\[
0 = V[\nabla^{0,1}_\sigma s] = V[\frac{1}{2}(Id + iI_\sigma)\nabla s] =
\frac{i}{2}V[I_\sigma]\nabla s + \nabla^{0,1}_\sigma V[s].
\]
If we combine the previous two equations we get the following
\begin{lemma}
  \label{lem_hitchin:1}
  The connection $\hat{\nabla} = \nabla^t + u$ preserves
  $H^{(k)}_\sigma$ for all $\sigma \in \T$ if and only if $u$ satisfy
  the equation
\begin{equation} \label{eq_hitchin:1}
\nabla^{0,1}_\sigma u(V) s = \frac{i}{2}V[I]_\sigma
\nabla^{1,0}_\sigma s
\end{equation}
for all $\sigma \in \T$ and all vector fields $V$ on $\T$.
\end{lemma}
If the conclusion is true the collection of subspaces $H^{(k)}_\sigma
\subset C^\infty (M,\cL^k)$ constitute a subbundle $H^{(k)}$ of
$\cH^{(k)}$.

Let us now assume that $\T$ is a complex manifold, and that the family $I$ is
holomorphic. First of all, $\nabla^{1,0}_\sigma s$ is a section of
$(T_\sigma)^*\otimes \cL^k$, so it is constant in the
$\bT_\sigma$-direction, which is why $V[I]_\sigma''\nabla^{1,0}_\sigma
s = 0$, and by holomorphicity $V''[I] = V[I]''$, so $V''[I]_\sigma
\nabla^{1,0}_\sigma s = 0$. Hence we can choose $u(V'') = 0$, and we
therefore only need to focus on $u$ in the $V'$-direction.

$u(V)$ should be a differential operator acting on sections of $\cL^k$,
and be related to $I$, so let us construct an operator from
$I$.

Given a smooth symmetric bivector field $B$ on $M$ we define a differential operator on
smooth sections of $\cL^k$ by
\[
\Delta_B = \nabla^2_{B} + \nabla_{\delta B},
\]
where $\delta B$ is the divergence of a symmetric bivector field
\[
\delta_\sigma(B) = \tr \nabla_\sigma B.
\]
$\nabla^2_B$ is defined by 
\[
\nabla^2_{X,Y} = \nabla_{X} \nabla_{Y} s - \nabla_{\nabla_{X}Y}s,
\]
which is tensorial in the vector fields $X$ and $Y$. Thus we can
evaluate it on a bivector field, and have thus defined $\Delta_B$.

Recall the bivector field $G(V)$ defined by $G(V) \cdot \omega =
V'[I]$. Using the above construction give a differential operator
$\Delta_{G(V)}: C^\infty(M,\cL^k) \to C^\infty(M,\cL^k)$. Locally $G(V)
= \sum_j X_j \otimes Y_j$, and
\begin{equation} \label{eq_hitchin:2}
\Delta_{G(V)} = \nabla^2_{G(V)} +\nabla_{\delta G(V)} = \sum_j
\nabla_{X_j}\nabla_{Y_j} + \nabla_{\delta(X_j)Y_j},
\end{equation}
since $\delta(X_j\otimes Y_j) = \delta(X_j)Y_j + \nabla_{X_j}Y_j$,
where $\delta(X)$ is the usual divergence of a vector field, which can
be defined in many ways e.g. in terms of the Levi-Cevitta connection
on $M_\sigma$ by $\delta(X) = \tr\nabla_\sigma X$. 

The second order differential operator $\Delta_{G(V)}$ is the
cornerstone in the construction of $u(V)$. The idea in Andersen's
construction is to calculate $\nabla^{0,1}\Delta_{G(V)}s$ and find
remainder terms, which cancel other terms such that $\Delta_{G(V)}$
with these correction terms satisfy equation~\eqref{eq_hitchin:1}. When
calculating $\nabla^{0,1}\Delta_{G(V)}s$ the trace of the curvature of
$M_\sigma$ show up -- that is the Ricci curvature
$\Ric_\sigma$. From Hodge decomposition $\Ric_\sigma =
\Ric_\sigma^H + 2i\partial_\sigma \bar{\partial}_\sigma F_\sigma$
where $\Ric^H$ is harmonic and $F_\sigma$ is the Ricci potential. As with the family of
K\"{a}hler structures the Ricci potentials $F_\sigma$ is a family of
Ricci potentials parametrized by $\T$ and can therefore be
differentiated along a vector field on $\T$.

Define $u \in \Omega^1(\T, \D(M,\cL^k))$ by
\begin{equation}
  \label{eq_hitchin:3}
  u(V) = \frac{1}{4k+2n}(\Delta_{G(V)} + 2 \nabla_{G(V) \cdot dF} + 4k
  V'[F]).
\end{equation}

\begin{theorem}[Andersen \cite{A5}] \label{HCE}
  Let $(M,\omega)$ be a compact prequantizable symplectic manifold
  with $H^1(M,\bR) = 0$ and first Chern class $c_1(M,\omega) = n
  [\frac{\omega}{2\pi}]$. Let $I$ be a rigid holomorphic family of
  K\"{a}hler structures on $M$ parametrized by a complex manifold
  $\T$. Then
\[
\hat{\nabla}_V = \nabla^t_V + \frac{1}{4k+2n}(\Delta_{G(V)} + 2
\nabla_{G(V) \cdot dF} +4 k V'[F])
\]
is a Hitchin connection in the bundle $H^{(k)}$ over $\T$.
\end{theorem}

\begin{remark}
  The condition $c_1(M,\omega) = n[\frac{\omega}{2\pi}] = n
    c_1(\cL)$ can be removed by switching to the metaplectic
    correction. Here we make the same construction but now a square
    root of the canonical bundle of $(M,\omega)$ is tensored onto
    $\cL^k$. Such a square root exists exactly if the second
    Stiefel--Whitney class is $0$ -- that is if $M$ is spin, see
    \cite{AGL} and also \cite{Ch}.
\end{remark}
\begin{remark}
   The condition $H^1(M,\bR) = 0$ is used to make the calculations
    in the proof easier, but there is no known examples of manifolds
    with $H^1(M,\bR) \neq 0$, which satisfy the remaining conditions
    where the Hitchin connection cannot be built in this way. An
    example is the torus $T^{2n}$ which we will study in much greater
    detail in Section~\ref{sec:abelian}.
\end{remark}

\begin{remark}
  \label{rem:3}
  By using Toeplitz operator theory, it can be shown that under some
  further assumptions on the family of K\"{a}hler structures the
  Hitchin connection is actually projectively flat. A proof of this
  can be found in \cite{Gammelgaard}.
\end{remark}

Suppose $\Gamma$ is a group which acts by bundle automorphisms of
$\cL$ over $M$ preserving both the Hermitian structure and the
connection in $\cL$. Then there is an induced action of $\Gamma$ on
$(M,\omega)$. We will further assume that $\Gamma$ acts on $\T$ and
that $I$ is $\Gamma$-equivariant. In this case we immediately get the
following invariance.

\begin{lemma}
  \label{lem:3}
  The natural induced action of $\Gamma$ on $\cH^{(k)}$ preserves the
  subbundle $H^{(k)}$ and the Hitchin connection.
\end{lemma}

We are actually interested in the induced connection $\Nablae$ in
the endomorphism bundle $\End(H^{(k)})$. Suppose $\Phi$ is a section
of $\End(H^{(k)})$. Then for all sections $s$ of $H^{(k)}$ and all
vector fields $V$ on $\T$, we have that
\[
(\Nablae_V \Phi)(s) = \Nabla_V \Phi(s) - \Phi(\Nabla_V(s)).
\]
Assume now that we have extended $\Phi$ to a section of
$\Hom(\cH^{(k)},H^{(k)})$ over $\T$. Then
\begin{equation}
  \label{eq:9}
  \Nablae_V \Phi = \Nablaet_V \Phi + [\Phi, u(V)],
\end{equation}
where $\Nablaet$ is the trivial connection in the trivial bundle
$\End(\cH^{(k)})$ over $\T$.

\section{Toeplitz operators on compact K{\"a}hler manifolds}
\label{BZdq}
In this section we discuss the Toeplitz operators on compact K\"{a}hler
manifolds $(M,\omega)$ with K\"{a}hler structures parametrized by a
smooth manifold $\T$ and their asymptotics as the level $k$ goes to infinity. 

For each $f \in C^\infty(M)$ we consider the differential operator
$M^{(k)}_f: C^\infty(M,\cL^k) \to C^\infty(M,\cL^k)$ given by 
\[ M_f^{(k)}(s) = fs \] for all $s \in H^0(M,\cL^k)$.

These operators act on $C^\infty(M,\cL^k)$ and therefore also on the
trivial bundle $\cH^{(k)}$, however they do not preserve the subbundle
$H^{(k)}$. There is however a solution to this, which is given by the
Hilbert space structure. Integrating the inner product of two sections
of $\cL^k$ against the volume form associated to the symplectic form
gives the pre-Hilbert space structure on $C^\infty(M)$
\[
\inner{s_1}{s_2} = \frac{1}{m!} \int_M (s_1,s_2)\omega^m.
\]This is not only a pre-Hilbert space structure on
$C^\infty(M,\cL^k)$ but also on the trivial bundle $\cH^{(k)}$ which is
of course compatible with the trivial connection in this bundle. This
pre-Hilbert space structure induces a Hermitian structure
$\inner{\cdot}{\cdot}$ on the finite rank subbundle $H^{(k)}$ of
$\cH^{(k)}$. The Hermitian structure $\inner{\cdot}{\cdot}$ on
$H^{(k)}$ also induces the operator norm on $\End(H^{(k)})$. By the
finite dimensionality of $H^{(k)}_\sigma$ in $\cH^{(k)}_\sigma$ we
have the orthogonal projection $\pi^{(k)}_\sigma: \cH^{(k)}_\sigma \to
H^{(k)}_\sigma$. From these projections we can construct the Toeplitz
operators associated to any smooth function $f \in C^\infty(M)$. It is
the operator $T^{(k)}_{f,\sigma}: \cH^{(k)}_\sigma \to H^{(k)}_\sigma$
defined by
\[
T^{(k)}_{f,\sigma}(s) = \pi^{(k)}_\sigma(fs)
\]
for any element $s \in \cH^{(k)}_\sigma$ and any point $\sigma \in
\T$. Since the projections form a smooth map $\pi^{(k)}$ from $\T$ to
the space of bounded operators in the $L_2$-completion of
$C^\infty(M,\cL^k)$ the Toeplitz operators are smooth sections
$T^{(k)}_f$ of the bundle of homomorphisms $\Hom(\cH^{(k)},H^{(k)})$ and restrict to
smooth sections of $\End(H^{(k)})$.

\begin{remark}
  \label{rem:2}
  It should be remarked that the above construction could be used for
  any Pseudo-differential operator $A$ on $M$ with coefficients in
  $\cL^k$ -- it can even depend on $\s$, and we will then consider it
  as a section of $\Hom(\cH^{(k)},H^{(k)})$. However when we consider
  their asymptotic expansions or operator norms, we implicitly
  restrict them to $H^{(k)}$ and consider them as sections of
  $\End(H^{(k)})$ -- or as $\pi^{(k)}A \pi^{(k)}$.
\end{remark}

We need the following two theorems on Toeplitz operators to
proceed. The first is due to Bordemann, Meinrenken and Schlichenmaier
(see \cite{BMS}).

\begin{theorem}[Bordemann, Meinrenken and Schlichenmaier]\label{BMS1}
For any $f\in C^\infty(M)$ we have that
\[\lim_{k\ra \infty}\|T_{f}^{(k)}\| = \sup_{x\in M}|f(x)|.\]
\end{theorem}

Since the association of the sequence of Toeplitz operators
$T^{(k)}_f$, $k\in \Z_+$ is linear in $f$, we see from this Theorem,
that this association is faithful.

The product of two Toeplitz operators associated to two smooth
functions will in general not be a Toeplitz operator associated to a
smooth function again. But by Schlichenmaier \cite{Sch}, there is an
asymptotic expansion of the product in terms of Toeplitz operators
associated to smooth functions on a compact K\"{a}hler manifold.

\begin{theorem}[Schlichenmaier]\label{S}
For any pair of smooth functions $f_1, f_2\in \C^\infty(M)$, we
have an asymptotic expansion
\[T_{f_1}^{(k)}T_{f_2}^{(k)} \sim \sum_{l=0}^\infty T_{c_l(f_1,f_2)}^{(k)} k^{-l},\]
where $c_l(f_1,f_2) \in C^\infty(M)$ are uniquely determined since
$\sim$ means the following: For all $L\in \Z_+$ we have that
\begin{equation}
\|T_{f_1}^{(k)}T_{f_2}^{(k)} - \sum_{l=0}^L T_{c_l(f_1,f_2)}^{(k)} k^{-l}\| =
O(k^{-(L+1)}) \label{normasympToep}
\end{equation}
uniformly over compact subsets of $\T$. Moreover, $c_0(f_1,f_2) = f_1f_2$.
\end{theorem}
\begin{remark} \label{rem:1} In Section~\ref{sec:form-hitch-conn} it will be useful for us to define new
  coefficients ${\tilde c}_\s^{(l)}(f,g) \in C^\infty(M)$ which
  correspond to the expansion of the product in $1/(k+n/2)$ (where $n$
  is some fixed integer):
  \[T_{f_1,\s}^{(k)}T_{f_2,\s}^{(k)} \sim \sum_{l=0}^\infty T_{{\tilde
      c}_\s^{(l)}(f_1,f_2),\s}^{(k)} (k+n/2)^{-l}.\] Note that the first three coefficients are given by
  $\tilde c^{(0)}_\sigma(f_1,f_2) = c^{(0)}_\sigma(f_1,f_2)$,\\ $\tilde
  c^{(1)}_\sigma(f_1,f_2) = c^{(1)}_\sigma(f_1,f_2)$ and $\tilde
  c^{(2)}_\sigma(f_1,f_2) = c^{(2)}_\sigma(f_1,f_2) + \frac{n}{2}
  c^{(1)}_\sigma(f_1,f_2)$.
\end{remark}

This Theorem was proved in \cite{Sch} where it
is also proved that the formal
generating series for the $c_l(f_1,f_2)$'s gives a formal
deformation quantization of the Poisson structure on $M$ induced by
$\omega$. An English version is available in \cite{Sch1} see
\cite{Sch2} for further developments. We return to this in Section~\ref{sec:form-hitch-conn} where we
discuss formal Hitchin connections. 

\section{Asymptotic faithfulness}
\label{sec:asympt-faithf}

In this section we will concentrate on the case where $M$ is the
moduli space of flat $\SU(n)$-connections on $\Sigma-p$ with
holonomy $d$ around $p$. As in the introduction $\Sigma$ is a closed
oriented surface of genus $g \geq 2$, $p$ a point in $\Sigma$ and $d
\in \Z/n\Z \simeq Z_{\SU(n)}$ in the center of $\SU(n)$ is fixed.

As mentioned in the introduction the main result about the
Verlinde bundle $\V^{(k)}$ from geometrically quantizing the moduli
space $M$ is that its projectivization $\P(\V^{(k)})$ carries a flat
connection $\Nabla$. This flat connection induces a flat connection in
$\Nablae$ in the endomorphism bundle $\End(\V^{(k)})$ as described in
Section~\ref{ghc}.

An important ingredient in proving asymptotic faithfulness is the
corollary to Theorem~\ref{MainFGHCI2} saying that Toeplitz operators
viewed as a section of $\End(\V^{(k)})$ is in
some sense asymptotically flat,
\[
\norm{\Nablae_VT^{(k)}_f} = O(k^{-1}).
\]
This can be reformulated in terms of the induced parallel transport
between the fibers of $\End(\V^{(k)})$. Let $\s_0,\s_1$ be two points
in Teichm\"{u}ller space $\T$, and $P_{\s_0,\s_1}$ the parallel transport
from $\s_0$ to $\s_1$. Then 
\begin{equation} \label{Asympflat}
\norm{P_{\s_0,\s_1}T^{(k)}_{f,\s_0} - T^{(k)}_{f,\s_1}} = O(k^{-1}),
\end{equation}
where $\norm{\cdot}$ is the operator norm on $H^0(M_{\s_1},\cL^k_{\s_1})$.

Equation~\eqref{Asympflat} and 
Theorem~\ref{BMS1} together prove asymptotic faithfulness. Below we
explain how.

Recall that the flat connection in the bundle $\P(\V^{(k)})$ gives the projective representation of the mapping
class group
\[
Z^{(n,d)}_k: \Gamma \to \Aut(\P(V^k)) 
\]
where $\P(V^{(k)})$ are the covariant constant sections of
$\P(\V^{(k)})$ over Teichm\"{u}ller space with
respect to the Hitchin connection $\Nabla$.

\begin{proof}[Proof of Theorem~\ref{MainA3}]
  Suppose we have a $\phi \in \Gamma$. Then $\phi$ induces a symplectomorphism of $M$
which we also just denote $\phi$ and we get the following commutative
diagram for any $f\in \C^\infty(M)$
\begin{equation*}
\begin{CD}
H^0(M_\sigma,\L_{\s}^k) @>\phi^*>> H^0(M_{\phi(\sigma)},\L_{\phi(\sigma)}^k)
@> P_{\phi(\sigma),\sigma}>> H^0(M_\sigma,\L_{\s}^k)\\
@V T^{(k)}_{f,\sigma} VV @V T^{(k)}_{f \circ \phi,\phi(\sigma)} VV
@VV{P_{\phi(\sigma),\sigma}T^{(k)}_{f \circ \phi,\phi(\sigma)}}V\\
H^0(M_\sigma,\L_{\s}^k) @>\phi^*>> H^0(M_{\phi(\sigma)},\L_{\phi(\sigma)}^k)
@> P_{\phi(\sigma),\sigma}>> H^0(M_\sigma,\L_{\s}^k),
\end{CD}
\end{equation*}
where $P_{\phi(\sigma),\sigma} : H^0(M_{\phi(\sigma)},\L_{\phi(\sigma)}^k) \ra H^0(M_{\sigma},\L_{\s}^k)$
on the horizontal arrows refer to parallel transport in the Verlinde
bundle itself, whereas  $P_{\phi(\sigma),\sigma}$  refers to the parallel transport in
the endomorphism bundle $\End(\V_k)$ in the last vertical arrow.
Suppose now $\phi \in \bigcap_{k=1}^\infty \ker
Z^{(n,d)}_k$, then $P_{\phi(\sigma),\sigma} \circ \phi^* = Z^{(n,d)}_k(\phi) \in
\bC \id$ and we get that $T^{(k)}_{f,\sigma} = P_{\phi(\sigma),\sigma}
T^{(k)}_{f\circ \phi,\phi(\sigma)}$. By Theorem \ref{Asympflat} we get
that
\[
\begin{split}
\lim_{k\ra \infty}\|T_{f - f \circ \phi,\sigma}^{(k)}\| &=
\lim_{k\ra \infty}\|T_{f,\sigma}^{(k)} - T_{f\circ \phi,\sigma}^{(k)}\|\\
& =  \lim_{k\ra \infty}\| P_{\phi(\sigma),\sigma}
T^{(k)}_{f\circ \phi,\phi(\sigma)} - T_{f\circ \phi,\sigma}^{(k)} \| = 0.
\end{split}\]
By Bordemann, Meinrenken and Schlichenmaier's Theorem \ref{BMS1},
we must have that $f = f \circ \phi$. Since this holds for any
$f\in \C^\infty(M)$, we must have that $\phi$ acts by the identity
on $M$.
\end{proof}

\section{The Formal Hitchin connection and Berezin--Toeplitz
  Deformation Quantization}
\label{sec:form-hitch-conn}

In this section we discuss the formal Hitchin connection. We return to
the general setup of compact K\"{a}hler manifolds, where we impose
conditions on $(M,\omega,I)$ as in Theorem~\ref{HCE}, thus
providing us with a Hitchin connection $\Nabla$ in $H^{(k)}$ over $\T$
and the associated connection $\Nablae$ in $\End(H^{(k)})$. Firstly werecall the definition of a formal deformation quantization and
the results about star products from \cite{Sch} and
\cite{KS}. We introduce the space of formal functions
$C^\infty_h(M) = C^\infty(M)\formal{h}$ as the space for
formal power series in the variable $h$ with coefficients in
$C^\infty(M)$, and let $\bC_h = \bC\formal{h}$ denote the formal constants.

\begin{definition}
  A deformation quantization of $(M,\omega)$ is an associative product
  $\star$ on $C^\infty_h(M)$ which respects the $\bC_h$-module
  structure.  For $f,g \in C^\infty(M)$, it is defined as
  \[f \star g = \sum_{l=0}^\infty c^{(l)}(f,g) h^{l},\] through a
  sequence of bilinear operators
$$c^{(l)} : C^\infty(M) \otimes C^\infty(M) \ra C^\infty(M),$$
which musth satisfy
\begin{align*}
  c^{(0)}(f, g) = fg \qquad \text{and} \qquad c^{(1)}(f, g) -
  c^{(1)}(g,f) = -i \{f, g\}.
\end{align*}
The deformation quantization is said to be \emph{differential} if the
operators $c^{(l)}$ are bidifferential operators. Considering the
symplectic action of $\Gamma$ on $(M,\omega)$, we say that a star
product is $\Gamma$-invariant if
$$ \gamma^*(f \star g) = \gamma^*(f) \star \gamma^*(g)$$
for all $f,g\in C^\infty(M)$ and all $\gamma\in\Gamma$.
\end{definition}

Recall Theorem~\ref{S} where the asymptotic expansion of the product
of two Toeplitz operators associated to smooth functions $f_1,f_2$ on
$M$ create maps $c_i(f_1,f_2) \in C^\infty(M)$. In \cite{Sch}
Schlichenmaier also showed that these maps generate a star product.
It was first in \cite{KS} Karabegov and Schlichenmaier showed that it
was a differentiable star product.

\begin{theorem}[Karabegov \& Schlichenmaier]\label{tKS1}
  The product $\BTstar_\s$ given by
  \[f \BTstar_\s g = \sum_{l=0}^\infty c_\s^{(l)}(f,g) h^{l},\] where
  $f,g \in C^\infty(M)$ and $c_\s^{(l)}(f,g)$ are determined by
  Theorem \ref{S}, is a differentiable deformation quantization of
  $(M,\omega)$.
\end{theorem}

\begin{definition}
  The Berezin-Toeplitz deformation quantization of the compact
  K{\"a}hler manifold $(M_\s,\omega)$ is the product $\BTstar_\s$.
\end{definition}

For the remaining part of this paper we let $\Gamma$ be a symmetry
group as in Section~\ref{ghc}, that is a group which acts by bundle
automorphisms on $\cL$ over $M$ preserving both the Hermitian
structure and the connection in $\cL$. Such a group has an induced
action on $(M,\omega)$. Note that $\Gamma$ in the case of moduli
spaces is the mapping class group of the surface.

\begin{remark} Let $\Gamma_\s$ be the $\s$-stabilizer subgroup of
  $\Gamma$.  For any element $\gamma\in \Gamma_\s$, we have that
  \[\gamma^*(T^{(k)}_{f,\s}) = T^{(k)}_{\gamma^*f,\s}.\]
  This implies the invariance of $\BTstar_\s$ under the
  $\s$-stabilizer $\Gamma_\s$.
\end{remark}

\begin{remark} Using the coefficients from Remark \ref{rem:1}, we
  define a new star product by
  \[f\tBTstar_\s g = \sum_{l=0}^\infty {\tilde c}_\s^{(l)}(f,g)
  h^{l}.\] Then
  \[f\tBTstar_\s g = \left( (f\circ \phi^{-1})\BTstar_\s (g\circ
    \phi^{-1})\right) \circ \phi\] for all $f,g\in C_h^\infty(M)$,
  where $\phi(h) = \frac{2h}{2 + nh}$.
\end{remark}

Recall from the introduction the definition of a formal connection in
the trivial bundle of formal functions. Theorem \ref{MainFGHCI2},
establishes the existence of a unique formal Hitchin connection,
expressing asymptotically the interplay between the Hitchin connection
and the Toeplitz operators.

We want to give an explicit formula for the formal Hitchin connection
in terms of the star product $\tilde \star^{BT}$. We recall that in
the proof of Theorem \ref{MainFGHCI2}, given in \cite{A5}, it is shown
that the formal Hitchin connection is given by
\begin{equation}
  \tD(V)(f) =  - V[F]f + V[F]\tBTstar f + h (E(V)(f) -
  H(V) \tBTstar f), \label{formalcon}
\end{equation}
where $E$ is the one-form on $\T$ with values in $\D(M)$ such that
\begin{align}
  \label{eq:2}
  T^{(k)}_{E(V)f} = \pi^{(k)} o(V)^* f \pi^{(k)} + \pi^{(k)} f o(V)
  \pi^{(k)},
\end{align}
and $H$ is the one form on $\T$ with values in $C^\infty(M)$ such that
$H(V) = E(V)(1)$. In \cite{AG} an explicit expression for the
operator $E(V)$ is found by calculating the adjoint of
\[
o(V) = -\frac{1}{4}(\Delta_{G(V)}+ 2 \Nabla_{G(V)\cdot dF} - 2 nV'[F]).
\]
This operator is essential in the proof of Theorem~\ref{HCE} since by
comparing the above equation with Equation~\ref{eq_hitchin:3} we see
that $u(V) = \frac{1}{k + n/2}o(V) - V'[F]$. 

\begin{theorem}
  \label{thm:1}
  The formal Hitchin connection is given by
  \begin{align*}
    D_V f &= V[f] - \frac{1}{4}h \Delta_{\tilde G(V)}(f) +
    \frac{1}{2}h \nabla_{\tilde G(V)dF}(f) + V[F] \tBTstar f - V[F] f
    \\ & \quad - \frac{1}{2} h(\Delta_{\tilde G(V)}(F) \tBTstar f +
    nV[F] \tBTstar f - \Delta_{\tilde G(V)}(F) f - n V[F]f)
  \end{align*}
  for any vector field $V$ and any section $f$ of $C_h$.
\end{theorem}



When we geometrically quantize a symplectic manifold, we have to
choose a polarization of the complexified tangent bundle, to reduce
the space upon the quantum operators act. This is equivalent to
choosing a compatible complex structure on the symplectic manifold,
hence making it K\"{a}hler. It is however quite unfortunate that the
quantum space then depend on the choice of K\"{a}hler structure. The
solution to this is the projectively flat Hitchin Connection, which by
parallel transport between the fibers of $\Hk$ give us a space of
quantum states as the covariant constant sections of $\P\Hk$, which does not depend on the chosen complex
structure. Instead of doing geometric quantization we could do
Berezin--Toeplitz deformation quantization. The created star product
$\BTstar_\s$ depend on the complex structure, and in the same spirit
as above we want to make all these star products equivalent to a star
product which does not depend on $\s$. This is the purpose of the
formal Hitchin connection.

If the Hitchin connection is projectively flat, then the induced
connection in the endomorphism bundle is flat and hence so is the
formal Hitchin connection by Proposition 3 of \cite{A5}.

Recall from Definition \ref{formaltrivi2} in the introduction the
definition of a formal trivialization. As mentioned there, such a
formal trivialization will not exist even locally on $\T$, if $D$ is
not flat. However, if $D$ is flat, then we have the following result
from \cite{A5}.

\begin{proposition}\label{Dfc} 
  Assume that $D$ is flat and that $\tD = 0$ mod $h$. Then locally
  around any point in $\T$, there exists a formal trivialization. If
  $H^1(\T,\bR) = 0$, then there exists a formal trivialization defined
  globally on $\T$. If further $H^1_\Gamma(\T,D(M)) = 0$, then we can
  construct $P$ such that it is $\Gamma$-equivariant.
\end{proposition}

An immediate corollary of Proposition~\ref{Dfc} is
\begin{corollary}
  \label{cl:2}
  If $\T$ is contractible, then any \emph{flat} formal connection admits
  a global formal trivialization that is $\Gamma$-equivariant.
\end{corollary}

In the proposition, $H^1_\Gamma(\T,D(M))$ refers to the
$\Gamma$-equivariant first \mbox{de Rham} cohomology of $\T$ with
coefficients in the real vector space $D(M)$ of differential
operators on $M$.
The first steps towards proving that this cohomology group
vanishes in the case where $M$ is the moduli space have been taken in \cite{AV1,AV2, AV3, Vi}. 

In \cite{AG} an explicit formula for $P$ up to first order is found. 

\begin{theorem}\label{Expliformula}
  The $\Gamma$-equivariant formal trivialization of the
  formal Hitchin connection exists to first order, and we have the
  following explicit formula for the first order term of $P$
$$P_\sigma^{(1)}(f) = \frac{1}{4} \Delta_\sigma(f) + i \nabla_{X''_F}(f),$$
where $X''_F$ denotes the (0,1)-part of the Hamiltonian vector field
for the Ricci potential, $F$.
\end{theorem}

Now suppose we have a formal trivialization $P$ of the formal Hitchin
connection $D$.  We can then define a new smooth family of star
products, parametrized by $\T$, by
\[f\star_\s g = P_\s^{-1}(P_\s(f) \tBTstar_\s P_\s(g))\] for all
$f,g\in C^\infty(M)$ and all $\s\in \T$. Using the fact that $P$ is a
trivialization, it is not hard to prove

\begin{proposition}\label{ftrdq}
  The star products $\star_\s$ are independent of $\s\in\T$.
\end{proposition}
This is done by simply differentiating $\star_\s$ along a vector field
on $\T$, see \cite{A5}.

Then, we have the following which is proved in \cite{A5}.

\begin{theorem}[Andersen]\label{general}
  Assume that the formal Hitchin connection $D$ is flat and
\[H^1_\Gamma(\T,D(M)) = 0,\] then there is a $\Gamma$-invariant
trivialization $P$ of $D$ and the star product
\[f\star g = P_\s^{-1}(P_\s(f) \tBTstar_\s P_\s(g))\] is independent
of $\s\in \T$ and $\Gamma$-invariant. If $H^1_\Gamma(\T,C^\infty(M)) =
0$ and the commutant of $\Gamma$ in $D(M)$ is trivial, then a
$\Gamma$-invariant differential star product on $M$ is unique.
\end{theorem}

In \cite{AG} the star product of Theorem~\ref{general} is identified
up to second order in $h$. 

\begin{theorem}\label{star}
  The star product $\star$ has the form
  \[f\star g = fg - \frac{i}{2} \{f,g\} h + O(h^2).\]
\end{theorem}

We observe that this formula for the first-order term of $\star$
agrees with the first-order term of the star product constructed by
Andersen, Mattes and Reshetikhin in \cite{AMR2}, when we apply the
formula in Theorem \ref{star} to two holonomy functions
$h_{\gamma_1,\lambda_1}$ and $h_{\gamma_2,\lambda_2}$:
\[h_{\gamma_1,\lambda_1}\star h_{\gamma_2,\lambda_2} =
h_{\gamma_1\gamma_2,\lambda_1\cup\lambda_2} - \frac{i}{2}
h_{\{\gamma_1,\gamma_2\},\lambda_1\cup\lambda_2} + O(h^2).\] We recall
that $\{\gamma_1,\gamma_2\}$ is the Goldman bracket (see \cite{Go2})
of the two simple closed curves $\gamma_1$ and $\gamma_2$.

A similar result was obtained for the abelian case, i.e. in the case
where $M$ is the moduli space of flat $U(1)$-connections, by the first
author in \cite{A2}, where the agreement between the star product
defined in differential geometric terms and the star product of
Andersen, Mattes and Reshetikhin was proved to all orders.

\section{Abelian varieties and $U(1)$-moduli space}
\label{sec:abelian}

In this chapter we will investigate all the previous mentioned objects in
the setting of principally polarized abelian varieties $M = V /
\Lambda$, where $V$ is a real vector space with a symplectic form $\omega$, $\Lambda$ a discrete
lattice of maximal rank such that $\omega$ is integral and unimodular
when restricted to $\Lambda$. Let now $\T$ be the space of complex
structures on $V$, which are compatible with $\omega$. Then for any $I
\in \T$, $M_I = (M,\omega,I)$ is an abelian variety. A prime example
of an abelian variety is the abelian moduli space. Here we let $\Sigma$ be a
closed surface of genus $g$, and $M$ be the moduli space of flat
$\U(1)$-connections on $\Sigma$. Then
\[
M = \Hom(\pi_1(\Sigma),\U(1)) = H^1(\Sigma,\bR) / H^1(\Sigma,\Z).
\]
There is the usual symplectic structure $\omega$ on $H^1(M,\bR)$ which
is of course integral and unimodular over the lattice $H^1(M,\Z)$. We
will return to this example below when we consider abelian
Chern--Simons theory.

In the following we will focus on $M_I$ being a principal polarized
abelian variety, where the compatible complex structures are
parametrized by $\T$.
There exists a symplectic basis $(\lambda_1, \dots, \lambda_{2n})$
over the integers for $\Lambda$ (e.g. \cite[p. 304]{GH}). Let $(x_1,
\dots, x_n,y_1, \dots, y_{2n})$ be the dual coordinates on $V$. Then
\[
\omega = \sum_{i=1}^n dx_i \wedge dy_i.
\]

Let $A$ be the automorphism group of $(\Lambda,\omega)$. Then $A$
injects into the symplectomorphism group of $(M,\omega)$, and by using
the symplectic basis $(\lambda_1, \dots, \lambda_{2n})$ we get an
identification $A \simeq \text{Sp}(2n,\Z)$. Notice that $A$ acts
on the principal polarized variety $M_I$.

Using the symplectic basis we can identify $\T$ with the Siegel Upper
Half Space
\[
\bH = \{Z \in M_{n,n}(\C) \, | \, Z = Z^T, \, \im(Z)> 0\}.
\]
For any $I \in \T$ we have that $(\lambda_1, \dots, \lambda_{n})$ is
a basis over $\C$ for $V$ with respect to $I$. Let $(z_1, \dots, z_n)$
be the dual complex coordinates on $V$ relative to the basis
$(\lambda_1, \dots, \lambda_n)$. The complex structure $I$ determines
and is determined by a unique $Z \in \bH$ such that
\[
z = x+ Zy.
\]
Since any $Z \in \bH$ gives a complex structure, say $I(Z)$,
compatible with the symplectic form, we have a bijective map $I: \bH
\to \T$ given by sending $Z \in \bH$ to $I(Z)$. For $Z \in \bH$ we use
the notation $X = \re(Z)$ and $Y = \im (Z)$.

For each $Z \in \bH$ we explicitly construct a prequantum line bundle on
$M_{I(Z)}$. We do that by providing a lift of the action $\Lambda$
action on $V$ to the trivial bundle $\tilde{\cL} = V \times \C$, such
that the quotient is the prequantum line bundle $\cL_Z$. We only need to specify a set of multipliers
$\{e_\lambda\}_{\lambda \in \Lambda}$ and a
Hermitian structure $h$. The multipliers are non-vanishing functions on
$V$ that are holomorphic with respect to $I(Z)$ and depend on $Z$. They
should furthermore satisfy the following functional equation
\[
e_{\lambda'}(v+\lambda) e_\lambda(v) = e_{\lambda'}(v)e_\lambda(v+\lambda') = e_{\lambda+\lambda'}(v), 
\] for all $\lambda,\lambda' \in \Lambda$. The action of $\Lambda$ on
$\tilde{\cL}$ is given by
\[
\lambda\cdot (v,z) = (v+\lambda,e_\lambda(z)),
\] for all $\lambda \in \Lambda$ and $(v,z)\in \tilde{\cL}$. For a
fixed basis of $\Lambda$ the functional equations determine the multipliers for all
$\lambda \in \Lambda$. For $I(Z)$ we choose the multipliers
\begin{align*}
  e_{\lambda_i}(z) &= 1, & i = 1,\dots , n, \\
  e_{\lambda_{n+i}}(z) &= e^{-2\pi i z_i - \pi i Z_{ii}}, & i= 1, \dots, n.
\end{align*}
 The constructed line bundle is denoted $\cL_Z$. If we
define $h(z) = e^{-2\pi y \cdot Y y}$, where $Z = X + i Y$, it will
define a Hermitian structure on $V \times \C$ by
$h(z)\inner{\cdot}{\cdot}_\C$ where $\inner{\cdot}{\cdot}_\C$ is the
standard inner product on $\C^n$. This function satisfies the functional equation
\[
h(z + \lambda) = \frac{1}{\abs{e_\lambda(z)}^2}h(z),
\]
and the inner product on $V\times \C$ is invariant under the action of
$\Lambda$ and hence induces a Hermitian structure
$\inner{\cdot}{\cdot}$ on $\cL_Z$. By general theory of abelian
varieties, e.g. \cite [Sect. 2.6]{GH}, a line bundle with the above
multipliers and Hermitian metric $(\cL_Z,\inner{\cdot}{\cdot})$ has
curvature $-2\pi i\omega$, and hence is a prequantum line bundle. Note
taht the prequantum condition in Definition~\ref{prequantumb} is
scaled with $2\pi$. We could just have used $2\pi \omega$ as the
symplectic structure. We choose the normalization at hand to make
later equations nicer.

The space of holomorphic sections of $\cL^k_Z$, $H^0(M_Z, \cL^k_Z)$
has dimension $k^n$, and as in the general theory they give a vector
bundle $H^{(k)}$ over $\bH$ by letting $H^{(k)}_Z = H^0(M_Z,\cL^k_Z)$.

 The
$L^2$-inner product on $H^0(M_Z,\cL_Z^k)$ is given by
\[
(s_1,s_2) = \int_{M_Z} s_1(z)\overline{s_2(z)}h(z)dxdy,
\]
for $s_1,s_2 \in H^0(M_Z,\cL^k_Z)$.

A basis for
the space of sections are the \emph{Theta functions},
\[
\theta_\alpha^{(k)}(z,Z) = \sum_{l\in \Z^n}e^{\pi i k (l+\alpha)\cdot
  Z(l+\alpha)}e^{2\pi i k (l+\alpha)\cdot z},
\] where $\alpha \in \frac{1}{k}\Z^n /
\Z^n$. The Theta functions satisfies the following heat
equation,
\[
\frac{\partial \theta^{(k)}_\alpha}{\partial Z_{ij}} = \frac{1}{4\pi
  i k} \frac{\partial^2 \theta^{(k)}_\alpha}{\partial z_i \partial z_j}.
\]

The geometric interpretation of this differential equation is a
definition of a
connection $\tilde{\nabla}$ in the trivial $C^\infty(\C^n)$-bundle
over $\bH$, by
\[
 \tilde{\nabla}_{\frac{\partial}{\partial Z_{ij}}} =
 \frac{\partial}{\partial Z_{ij}} - \frac{1}{4\pi i k}
 \frac{\partial^2}{\partial z_i \partial z_j}.
\]
The coordinates $z = x+Zy$ identify
$H^0(M_Z,\cL^k_Z)$ as a subspace of $C^\infty(\C^n)$ and $H^{(k)}$ as a
subbundle of the trivial $C^{\infty}(\C^n)$-bundle on $\bH$. This
bundle is preserved by $\tilde{\nabla}$ and hence induces a
connection $\nabla$ in $H^{(k)}$. The covariant constant sections of
$H^{(k)}$ with respect to $\nabla$ will, under the embedding induced
by the coordinates, be identified with the Theta functions. Since
now $\nabla$ has a global frame of covariant constant sections it
is flat. Remember that $\bH$ is contractible, so since parallel
transport with a flat connection only depend on the homotopy class of
the curve transported along, we get a canonical way to identify all
$H^0(M_Z,\cL^k_Z)$, and hence there is no ambiguity in defining the
quantum space of geometric quantization to be
$H^0(M_Z,\cL^k_Z)$. Since the Theta functions are covariant constant,
they explicitly realize this identification. The usual action of
$\text{Sp}(2n,\Z)$ on Theta functions induce an action of $A' =
\ker(\text{Sp}(2n,\Z) \to \text{Sp}(2n,\Z/2\Z))$ on the
bundle $\Hk$ which covers the $A'$-action on $\bH \simeq \T$. This is
the subgroup of $A$ acting trivially on $\Lambda /2\Lambda$.

\begin{remark}
Instead of the above connection $\nabla$ in $\Hk$ over
$\bH$, we could have rolled out the machinery of Theorem~\ref{HCE} to
get another connection in the same bundle. This can be done even
though $H^1(M,\bR) \neq 0$. Since the torus is flat the Ricci potential
$F$ is $0$ as is the Chern class of $M$. Lemma~\ref{lem:2} in the
appendix shows that $I(Z)$ is constant on $M$
and thus is a rigid family of K\"{a}hler structures. Thus we have a
rather nice formula for the Hitchin connection
\[
\Nabla_V = \nabla^t_V + \frac{1}{8\pi k}\Delta_{G(V)}.
\]
The extra factor of $2\pi$ is from the different prequantum condition. It should be noted that explicit computations show that $\Nabla$ is
not flat like $\nabla$ induced by the heat equation, but rather
projectively flat. 
\end{remark}

In \cite{A2} the inner product of two Theta functions are explicitly
calculated.
\begin{lemma}
  \label{lem:4}
    The theta functions $\theta^{(k)}_\alpha(z,Z)$, $\alpha \in
  \frac{1}{k}\Z^n / \Z^n$, define an orthonormal basis with respect to
  the inner product on $H^0(M_Z,\cL^k_Z)$ defined by
  \[
  (s_1,s_2)_Y = (s_1, s_2) \sqrt{2^nk^n \det{Y}},
  \]
  where $Y = \im Z$. This is a Hermitian structure on $\Hk$ compatible
  with $\nabla$.
\end{lemma}

Let $(r,s) \in \Z^n \times \Z^n$ and consider the function $F_{r,s} \in
\C^\infty(M)$ given in $(x,y)$-coordinates by 
\[
F_{r,s}(x,y) = e^{2\pi i (x \cdot r + s \cdot y)}. 
\]
We have previously defined Toeplitz operators associated to a function
$f\in C^\infty(M)$, $T^{(k)}_f: H^0(M_Z, \cL_Z^k) \to
H^0(M_Z,\cL^k_Z)$. We shall now
explicitly compute the matrix coefficients of these operators in terms
of the basis consisting of Theta functions. 

To get our hands on the matrix coefficients $(T^{(k)}_{F_{r,s}})_{\beta\alpha}$ we only need to calculate
$(F_{r,s}\theta^{(k)}_\alpha,\theta^{(k)}_\beta)$, since this indeed
is the coefficient. This is also calculated in
\cite{A2} and is done in the exact same way as in Lemma~\ref{lem:4},
\begin{equation}
  \label{eq_dq_gq_ab:1}
  (F_{r,s}\theta^{(k)}_\alpha,\theta^{(k)}_\beta)_Y =
  \delta_{\alpha-\beta,-[\frac{r}{k}]}e^{-\frac{\pi i}{k}r\cdot
    \bar{Z}r}e^{-2\pi i s \cdot \alpha} e^{-\pi^2(s-\bar{Z}r)\cdot
    (2\pi k Y)^{-1}(s-\bar{Z}r)},
\end{equation}
where $[\frac{r}{k}]$ is the residue class of $\frac{r}{k}$ mod
$\Z^n$. A simple rewriting gives
\[
(T^{(k)}_{f(r,s,Z)(k)F_{r,s}})_{\beta\alpha} =
\delta_{\alpha-\beta,-[\frac{r}{k}]}e^{-\frac{\pi i}{k} r\cdot
  s}e^{-2\pi i s \cdot \alpha},
\] where
\[
f(r,s,Z)(k) = e^{\frac{\pi}{2k}(s-Xr)\cdot
  Y^{-1}(s-Xr)}e^{\frac{\pi}{2k}r \cdot Yr}.
\]
\begin{remark}
  \label{rem:4}
  The Toeplitz operators $T^{(k)}_{F_{r,s}}$ are sections of
  $\End(\Hk)$ over $\bH$. The flat connection $\nabla$ induces a flat
  connection $\nabla^e$ in the bundle $\End(\Hk)$, with
  respect to which we see that $T^{(k)}_{F_{r,s}}$ is \emph{not}
  covariant constant. However the operators
  $T^{(k)}_{f(r,s,Z)(k)F_{r,s}}$ are covariant constant. Since the
  pure phases $F_{r,s}$, $r,s \in \Z^n$ is a Fourier basis for
  $C^\infty(M)$, we have that $T^{(k)}_{f(r,s,Z)(k)F_{r,s}}$ is covariant constant
  with respect to $\Nablae$ for all $f \in C^\infty(M)$.
\end{remark}

It should also be noted that the coefficient $f(r,s,Z)(k)$ is not so
arbitrary as it looks. This is the content of the following
\begin{proposition}
  \label{prop:1}
  Let $\Delta_{I(Z)}$ be the Laplace operator with respect to the
  metric
\[
g_{I(Z)}(\cdot,\cdot) = 2\pi \omega(\cdot, I(Z) \cdot)
\] on $M$. Then
\[
e^{-\frac{1}{4k}\Delta_{I(Z)}}F_{r,s} = f(r,s,Z)(k)F_{r,s}.
\]
\end{proposition}
\begin{proof}
  Recall that 
\[
\Delta_{I(Z)} = \frac{1}{2\pi} \biggl( (\dery{} - X \derx{}) \cdot
Y^{-1}(\dery{} - X \derx{}) + \derx{} \cdot Y \derx{} \biggr).
\]
Now it is a simple calculation, which we will omit, to show the equality.
\end{proof}

As remarked in Remark~\ref{rem:4}, $T^{(k)}_{f(r,s,Z)(k)F_{r,s}}$ is
covariant constant with respect to $\Nabla$. If we define 
\[ E_{I(Z)} = e^{-\frac{h}{4}\Delta_{I(Z)}}: C^\infty_h(M) \to C^\infty_h(M) \]
we see that
\[
\Nablae_V T^{(k)}_{E_I(f)(1/k)} = 0,
\] for all vector fields $V$ on $\bH$ and all functions $f \in
C^\infty(M)$ since the pure phase functions consitute a Fourier basis. We furthermore see that $E_I$ is
$\text{Sp}(2n,\Z)$-equivariant, since for all $\Psi \in \text{Sp}(2n,\Z)$
we have that
\[
\Psi^* \circ E_I = E_{\Psi(I)}\Psi^*.
\]

That $T^{(k)}_{E_I(f)}$ is covariant constant with respect to
$\Nablae$ can be interpreted as $E_I$ is a formal parametrization for
the formal Hitchin Connection which we know exists by
Theorem~\ref{MainFGHCI2}. If $\Nablae_V T^{(k)}_{E_I(f)} = 0$
Equation~\eqref{Tdf2} and Theorem~\ref{BMS1} imply that
\[
D_V(E_I(f)) = 0
\]
for all vector fields on $\bH$ and all $f \in C^\infty_h(M)$, so by
Definition~\ref{formaltrivi2}, $E_I$ is a formal trivialization of the
formal connection $D$. We compare this with the explicit formula for
the first order term of $P$ in Theorem~\ref{Expliformula}, and see
that they agree since the Ricci potential $F$ is $0$.

Now since the Ricci potential is $0$ we reduce the formula in
Theorem~\ref{thm:1} for the formal Hitchin connection.

\begin{theorem}
  \label{thm:4}
  Let $(M,\omega,I(Z))$ be a principal polarized variety, then the formal Hitchin connection is given by
  \[
  D_Vf = V[f] - \frac{1}{8\pi}h\Delta_{\tilde{G}(V)}(f),
\] and if $Z$ is normal, we get explicit formulas for
$\Delta_{\tilde{G}(V)}$. If $i \neq j$
\[
\Delta_{\tilde{G}(\dert{ij})} = 2i \nabla_{\derz{i}}\nabla_{\derz{j}} + 2i
\nabla_{\derz{j}}\nabla_{\derz{i}} \enskip \text{and} \enskip
\Delta_{\tilde{G}(\derbt{ij})} = 2i \nabla_{\derbz{i}}\nabla_{\derbz{j}} + 2i \nabla_{\derbz{j}}\nabla_{\derbz{i}},
\]
and if $i = j$
\[
\Delta_{\tilde{G}(\dert{ii})} = 2i \nabla_{\derz{i}}\nabla_{\derz{i}} \quad
\text{and} \quad \Delta_{\tilde{G}(\derbt{ii})} = 2i \nabla_{\derbz{i}}\nabla_{\derbz{i}}.
\]
\end{theorem}

This theorem is proved in the appendix. It should be noted that the
requirement on $Z$ to be normal, only is to ease the calculations, and
it will not be used anywhere else in the rest of this paper.

With this formal trivialization we use Theorem~\ref{general} and create an $I$ independent star
product on $C^\infty(M)$ which all Berezin--Toeplitz star products are
equivalent to. This is done in \cite{A2} Theorem~5 where it is shown
that the $I$ independent star product actually is the Moyal--Weyl
product
\[
f \star g = \mu \circ \exp(-\frac{i}{2}h Q)(f \otimes g),
\]
where $\mu: C^\infty(\bR^{2n}) \otimes C^\infty(\bR^{2n}) \to
C^\infty(\bR^{2n})$ given by multiplication $f\otimes g \mapsto fg$ and
\[
Q = \sum_i \derx{i} \otimes \dery{i} - \dery{i} \otimes \derx{i}.
\]
Again we see that this is exactly as in Theorem~\ref{star}.

\subsection*{Abelian Chern--Simons Theory}
\label{sec:abelian-chern-simons}

In $2+1$ dimensional Chern--Simons theory, the $2$-dimensional part of
the theory is a modular functor, which is a functor from the category
of compact smooth oriented surfaces to the category of finite
dimensional complex vector spaces, which satisfy certain
properties. In the gauge-theoretic construction of this functor one
first fixes a compact Lie group $G$ and an invariant non-degenerate
inner product on its Lie algebra. The functor then associates to a
closed oriented surface the finite dimensional vector space one obtains
by applying geometric quantization to the moduli space of flat
$G$-connections on the surface (see e.g. \cite{W1} and \cite{At}). In
the abelian case $G = U(1)$ at hand this concretely means the
following. For a closed oriented surface $\Sigma$ the moduli space of
flat $U(1)$-connections
\[
M= \Hom(\pi_1(\Sigma),U(1)) = H^1(\Sigma,\bR)/H^1(\Sigma,\Z)
\]
has a symplectic structure given by the cup product followed by
evaluation on the fundamental class of $\Sigma$. This symplectic
structure is by Poincaré duality integral and
is unimodular over the lattice $H^1(\Sigma,\Z)$. A subgroup of the mapping class
group $\Gamma$ of $\Sigma$ acts on $M$ via the induced homomorphism
\[
\rho: \Gamma \to \Aut(H^1(\Sigma,\Z),\omega) = \text{Sp}(2n,\Z).
\] Define $\Gamma' = \rho^{-1}(A')$ and 
\[
\rho'=\rho |_{\Gamma'} : \Gamma' \to A'.
\]
 The homomorphism $\rho'$ is surjective and has the Torelli subgroup
of $\Gamma$ as its kernel.

If we use the above theory we construct a Hermitian
vector bundle $\Hk$ over the space of complex structures $\T$ on
$H^1(\Sigma,\bR)$. As discussed this bundle has a flat connection, and
an action of action of $\Aut(H^1(\Sigma,\Z),\omega)$ that preserves
the Hermitian structure and the flat connection. In
this case the modular functor is defined by associating to $\Sigma$,
the vector space $Z^{(k)}(\Sigma)$ consisting of covariant constant
sections of $\Hk$ over $\T$. So through the representation $\rho$, we
get a representation $\rho_k$ of the mapping class group $\Gamma$ of
$\Sigma$ on $Z^{(k)}(\Sigma)$. In the $\SU(n)$-case in the
introduction this representation was denoted $Z^{(n,d)}_k$. 

The $2+1$ dimensional Chern--Simons theory also fits into a TQFT
setup. Suppose $Y$ is a compact oriented $3$-manifold such that
$\partial Y = (-\Sigma_1) \cup \Sigma_2$, where $\Sigma_1$ and
$\Sigma_2$ are closed oriented surfaces and $-\Sigma_1$ is $\Sigma_1$
with reversed orientation. Assume furthermore that $\gamma$ is a link
inside $Y - \partial Y$. Then the TQFT-axioms states that there should
be a linear morphism $Z^{(k)}(Y,\gamma): Z^{(k)}(\Sigma_1) \to
Z^{(k)}(\Sigma_2)$, which satisfies that gluing along boundary
components goes to the corresponding composition of linear maps.
\begin{definition}
  \label{def:1}
  The curve operator 
  \[
  Z^{(k)}(Y,\gamma): Z^{(k)}(\Sigma_1) \to Z^{(k)}(\Sigma_2),
  \]
  is defined to be
  \[
  Z^{(k)}(Y,\gamma):= T^{(k)}_{E_{I(Z)}(h_{\gamma}),I(Z)},
  \]
  where $h_{\gamma}$ is the holonomy function associated to
  $\gamma$. 
\end{definition}
To a simple closed curve $\gamma$ on $\Sigma$ the holonomy function
$h_\gamma \in C^\infty(M)$ is a pure phase function, i.e. $h_\gamma =
F_{r,s}$ where $r,s \in \Z^n$. Note that we do not label $\gamma$ with an
irreducible $U(1)$-representation $\lambda$.

Using this definition we could give the exact same proof as of Theorem~\ref{MainA3} and obtain a classical theorem from the
theory of theta functions.
\begin{theorem}
  \label{thm:2}
  Elements in the Torelli subgroup $\ker \rho'$ are
  exactly those who are in the kernels of all $\rho_k's$,
\[
\bigcap_{k=1}^\infty \ker \rho_k = \ker \rho'.
\]
\end{theorem}

To this end we want to give a proof of Theorem~$1$ from \cite{A8} in
the case of abelian moduli spaces. We do this by studying the Hilbert--Schmidt norm of the curve
operators.
\begin{definition}
  \label{def:2}
  The Hilbert--Schmidt inner product of two operator $A,B$ is
  \[
  \inner{A}{B} = \tr(AB^*).
\]
\end{definition}
If we introduce the notation
\begin{align*}
\eta_k(r,s) &= \re(e^{-\frac{\pi i}{k}r\cdot
    \bar{Z}r}e^{-2\pi i s \cdot \alpha} e^{-\pi^2(s-\bar{Z}r)\cdot
    (2\pi k Y)^{-1}(s-\bar{Z}r)})\\  &= e^{-\frac{\pi}{2k}((s-Xr)\cdot
    Y^{-1}(s-Xr) + r \cdot Y r)}
\end{align*} and recall the matrix coefficients of the Toeplitz operators
$T^{(k)}_{F_{r,s}}$ in terms of the basis of theta functions then 
\[
(F_{r,s}\theta^{(k)}_{\alpha},\theta^{(k)}_{\beta})_Y =
\delta_{\alpha-\beta,-[\frac{r}{k}]} e^{-2\pi i s \cdot
  \alpha}e^{-\frac{\pi i}{k} r \cdot s} \eta_k(r,s).
\]
Note that $f(r,s,Z)(k) =
\eta_k(r,s)^{-1}$. Here we suppress the $Z$ dependence in
$\eta_k(r,s)$ since we from now only consider fixed K\"{a}hler
structure.

\begin{lemma}
  \label{lem:1}
  \[
  \tr(T^{(k)}_{F_{r,s}}(T^{(k)}_{F_{t,u}})^*) =
  \begin{cases}
    k^n \eta_k(r,s)
    \eta_k(t,u) \epsilon(r,s,t,u) & (r,s) \equiv (t,u) \text{ mod } k \\
    0  & else
  \end{cases}
  \]
where $\epsilon(r,s,t,u) \in \{\pm 1\}$ and is $1$ for $(r,s) = (t,u)$.
\end{lemma}
\begin{proof}
  We start by calculating the matrix coefficients of the product of
  the Toeplitz operators
  \begin{align*}
    (T^{(k)}_{F_{r,s}}(T^{(k)}_{F_{t,u}})^*)_{\beta\alpha} &= \sum_\phi
    (T^{(k)}_{F_{r,s}})_{\beta\phi}
    \overline{(T^{(k)}_{F_{t,u}})_{\alpha\phi}} \\
&= \delta_{\alpha-\beta,-[\frac{r-t}{k}]}e^{-2\pi i \alpha \cdot
  (s-u)} e^{-\frac{\pi i}{k}(r\cdot s - 2 s\cdot t + t\cdot u)} \eta_k(r,s)\eta_k(t,u).
  \end{align*}
Now when taking the trace $\alpha = \beta$ and to get something
non-zero we must have $r \equiv t$ mod $k$. In that case
\[
\tr(T^{(k)}_{F_{r,s}}(T^{(k)}_{F_{t,u}})^*) =
\epsilon(r,s,t,u)\eta_k(r,s)\eta_k(t,u)e^{\frac{\pi i}{k}r \cdot (s-u)} \sum_\alpha
e^{-2\pi i \alpha(s-u)},
\] the $\epsilon$ is obtained since $t = r + kv$ only determines the
equality
\[ e^{-\frac{\pi i}{k}(r\cdot s - 2 s\cdot t + t\cdot u)} = \pm
e^{\frac{\pi i}{k}(r \cdot(s-u))} \] up to a sign.
Now if $s \not\equiv u$ the last term is zero since it is $n$ sums of
all $k$'th roots of unity, and hence $0$. If $s \equiv u$ each term in
the sum is $1$, and we get the desired result.
\end{proof}

Using the above lemma and the following limits
\begin{equation}\label{eq:limits}
\lim_{k \to \infty} \eta_k(r,s)  = 1 \qquad \text{and} \qquad \lim_{k
  \to \infty} \eta_k(r+kt,s+ku) = 0, 
\end{equation} for all $r,s \in \Z^n$, we can prove the following

\begin{theorem}
  \label{thm:3}
  For any two smooth functions $f,g \in C^\infty(M)$ and any $Z \in
  \bH$ one has that
  \[
  \inner{f}{g} = \lim_{k\to\infty} k^{-n}
  \inner{T^{(k)}_{f,I(Z)}}{T^{(k)}_{g,I(Z)}},
\] where the real dimension of $M$ is $2n$.
\end{theorem}

\begin{proof}
From Lemma~\ref{lem:1} we get in particular
\[
\norm{T^{(k)}_{F_{r,s}}}_k = k^{-n/2}
\sqrt{\tr(T^{(k)}_{F_{r,s}}(T^{(k)}_{F_{r,s}})^*)} = \eta_k(r,s),
\]
and
\[
\norm{T^{(k)}_{E_I(F_{r,s})}}_k = 1,
\]
where $\norm{\cdot}_k = k^{-n/2}\sqrt{\inner{\cdot}{\cdot}}$ is the
$k$-scaled Hilbert--Schmidt norm.

Let $f,g \in C^\infty(M)$ be an arbitrary elements and expand them in 
Fourier series
\[
f = \sum_{(r,s) \in \Z^{2n}} \lambda_{r,s} F_{r,s} \qquad
\text{and} \qquad g = \sum_{(t,u) \in \Z^{2n}} \mu_{t,u} F_{t,u}.
\]
$\eta_k(r,s)$ and $\eta_k(t,u)$ decays very fast for increasing $r,s
\in \Z^n$ and we have
\begin{align*}
k^{-n} \tr(T^{(k)}_f& (T^{(k)}_g)^*) = k^{-n} \sum_{(r,s), (t,u) \in
  \Z^{2n}} \lambda_{r,s}\bar{\mu}_{t,u}
\tr(T^{(k)}_{F_{r,s}}(T^{(k)}_{F_{t,u}})^*) \\
=& \sum_{(r,s) \in
  \Z^{2n}}\lambda_{r,s}\bar{\mu}_{t,u}\eta_k(r,s)^2 \\
&+ \sum_{\substack{(r,s),(t,u) \in \Z^{2n}\\ (t,u) \neq (0,0)}}
\lambda_{r,s}\bar{\mu}_{r+kt,s+ku} \eta_k(r,s)\eta_k(r+kt,s+ku)\epsilon(r,s,t,u).
\end{align*}
This sum converges uniformly so if we take the large $k$ limit we can
interchange limit and summation. Now by Equation~\ref{eq:limits} and
since
\[
\lim_{k \to \infty}\mu_{r+kt,s+ku} = 0
\] by pointwise convergence of the Fourier series we finally get
\[
\lim_{k \to \infty} k^{-n} \tr(T^{(k)}_f (T^{(k)}_g)^*) = \sum_{(r,s)
  \in \Z^{2n}} \lambda_{r,s} \bar{\mu}_{r,s}.
\]
Now since the pure phase functions are orthogonal we get to desired result.
\end{proof}

It should be remarked that Theorem~\ref{thm:3} just is a particular
case of a theorem of the same wording, with $M$ being a compact
K\"{a}hler manifold, see e.g. \cite{A8}. Theorem~\ref{thm:3} was also
proved in \cite{BHSS} but only for a small class of principal
polarized abelian varieties.

As a corollary to the proof of Theorem~\ref{thm:3} we have
\begin{corollary}
  \label{cl:1}
\[
  \inner{f}{g} = \lim_{k \to \infty} k^{-n} \inner{T^{(k)}_{E_I(f)}}{T^{(k)}_{E_I(g)}}.
\]
\end{corollary}

We can interpret Corollary~\ref{cl:1} in terms of TQFT curve
operators. Since we defined a curve operator $Z^{(k)}(\Sigma,\gamma)$
to be $T^{(k)}_{E_I(h_\gamma)}$ where $h_\gamma$ is the corresponding
holonomy function of $\gamma$ we immediately get
\[
\inner{h_{\gamma_1}}{h_{\gamma_2}} = \lim_{k \to
  \infty}k^{-n}\inner{Z^{(k)}(\Sigma,\gamma_1)}{Z^{(k)}(\Sigma,\gamma_2)},
\] which was proved in \cite{A8} and \cite{MN}.

Another interpretation is that gluing two cylinders
$(\Sigma\times[0,1],\gamma_1)$ and $(\Sigma\times[0,1],\gamma_2)$
along $\Sigma \times \{0\}$ and $-\Sigma \times \{0\}$ and again at the
top $\Sigma \times \{1\}$ along $-\Sigma \times \{1\}$, we obtain the
closed three manifold $\Sigma \times S^1$ with the link $\gamma_1 \cup
\gamma_2^*$ embedded. Here $\gamma_2^*$ means $\gamma_2$ with reversed
orientation. The TQFT gluing axioms now say that
\[
Z^{(k)}(\Sigma\times S^1,\gamma_1 \cup \gamma_2^*) =
\tr(Z^{(k)}(\Sigma\times[0,1],\gamma_1) Z^{(k)}(\Sigma\times[0,1],\gamma_2)^*).
\]
If we now define $Z^{(k)}(\Sigma \times S^1, \gamma_1 \cup
\gamma_2^*)$ to be exactly this, we see that if we take $\gamma_1$ and
$\gamma_2$ to be the empty links we have
\[
Z^{(k)}(\Sigma \times S^1) = \tr(T^{(k)}_{1}(T^{(k)}_{1})^*) = k^n = \dim(Z^{(k)}(\Sigma)),
\]
as it should be according to the axioms.

\section{Appendix}

In this appendix we provide the calculations needed to prove the
explicit formuli for the formal Hitchin connection given in Theorem~\ref{thm:4}.

We first observe that the theorem will follow from
Equation~\ref{eq_hitchin:2} if we can show that for $i \neq j$
\[
\tilde{G}(\dert{ij}) = 2i \derz{i} \otimes \derz{j} + 2i \derz{j}
\otimes \derz{i} \quad \text{and} \quad \tilde{G}(\derbt{ij}) = 2i \derbz{i} \otimes \derbz{j} + 2i \derbz{j}
\otimes \derbz{i}
\]
and for $i = j$
\[
\tilde{G}(\dert{ii}) = 2i \derz{i} \otimes \derz{i} \quad \text{and}
\quad \tilde{G}(\derbt{ii}) = 2i \derbz{i} \otimes \derbz{i}.
\]
Since the family of K\"{a}hler structures parametrized by $\bH$ is
holomorphic we just have to solve the equations
\[
G(\dert{ij}) \cdot \omega = \frac{\partial I(Z)}{\partial Z_{ij}}
\quad \text{and} \quad \bar{G}(\derbt{ij}) \cdot \omega =
\frac{\partial I(Z)}{\partial \bar{Z}_{ij}}.
\]

\begin{lemma}
  \label{lem:2}
  The K\"{a}hler structure associated to a $Z = X + i Y \in \bH$ is
\[  I(Z) =
\begin{pmatrix}
  -Y^{-1}X & -(Y+XY^{-1}X) \\
  Y^{-1} & XY^{-1}
\end{pmatrix},
\]where we have written it as tensor in the frame $\derx{i},\dery{j}$
of the tangent bundle $TM$. 
\end{lemma}
\begin{proof}
  This follows from the fact that the complex frame of $TM$ are
  eigenvectors for $I(Z)$, that is 
\[
I(Z)(\derz{i}) = i \derz{i} \quad \text{and} \quad I(Z)(\derbz{i}) =
-i \derbz{i},
\]
and that
 \begin{align*}
    \derz{\,} &= \frac{i}{2} Y^{-1} \bar{Z} \derx{} +
    \frac{1}{2i}Y^{-1} \dery{} =
    \frac{i}{2}(Y^{-1}X-i)\derx{} + \frac{1}{2i}Y^{-1}\dery{} \\
    \derbz{} &= \frac{1}{2i} Y^{-1}Z \derx{} +
    \frac{i}{2}Y^{-1}\dery{} = \frac{1}{2i}(Y^{-1}X+i)\derx{}
    +\frac{i}{2}Y^{-1}\dery{}.
  \end{align*}
\end{proof}
In the above we used the vector notation $\derz{}$ meaning an $n$-tuple
of vectors $\derz{i}$. This convention eases the following
calculations and will be used in the following.

\begin{proof}[Proof of Theorem~\ref{thm:4}]
To this end we also need to recall the following derivation property
for matrices. If $A = (a_{ij})$ is a symmetric invertible $n\times
n$-matrix then
\[
\frac{\partial A^{-1}}{\partial a_{ij}} = - A^{-1} \frac{\partial
  A}{\partial a_{ij}} A^{-1} = -A^{-1} \Delta_{ij} A^{-1},
\]
where $\Delta_{ij}$ is an $n\times n$-matrix with all entries $0$
except the $ij$'th and $ji$'th which is $1$, if $i\neq j$ and
$\Delta_{ii}$ is an $n\times n$-matrix with all entries $0$ except the
$ii$'th diagonal entry which is $1$. This follows easily from $A^{-1}A =
Id$. Using this rule and that $Y^{-1} = 2i (Z-\bar{Z})^{-1}$ we get
\[
\frac{\partial Y^{-1}}{\partial Z_{ij}} = - \frac{1}{2i} Y^{-1}\Delta_{ij}Y^{-1}.
\]
Derivation of the above equations with
respect to $Z_{ij}$ becomes rather messy if we do not also
require $Z$ to be normal, that is since $Z$ is symmetric $[Z,\bar{Z}]=
0$, which is equivalent to $[X,Y] = 0$, or $[X,Y^{-1}]=0$. A
consequence of this is, that everything will commute even
$[Y^{-1},\Delta_{ij}] = 0$ since the imaginary part of derivation of $Z \bar{Z} = \bar{Z} Z$
with respect to $Z_{ij}$ give $Y\Delta_{ij} = \Delta_{ij} Y$, and
hence $[Y^{-1},\Delta_{ij}] = 0$. Written as a tensor
\[
\frac{\partial I(Z)}{\partial Z_{ij}} =
\frac{1}{2i}Y^{-1}\Delta_{ij}Y^{-1}
\begin{pmatrix}
  \bar{Z} & \bar{Z}^2 \\ -1 & -\bar{Z}
\end{pmatrix}.
\]
The symplectic form $\omega
= -\frac{1}{2i} \sum_{ij = 1}^n w_{ij} dz_i \wedge~d\bar{z}_j$ where
$Y^{-1} = W = (w_{ij})$, should be contracted with $G(\dert{ij})$ we want to know its appearance
in the $Z$-dependent $\derz{}$, $\derbz{}$ frame. It is clear from above that 
\[
\frac{\partial I(Z)}{\partial Z_{ij}}(\derz{}) =
\frac{i}{2}Y^{-1}\bar{Z} \frac{\partial I(Z)}{\partial Z_{ij}}(\derx{})
+ \frac{1}{2i}Y^{-1}\frac{\partial I(Z)}{\partial Z_{ij}}(\dery{}) = 0,
\]
and an easy calculation shows that
\[
\frac{\partial I(Z)}{\partial Z_{ij}} (\derbz{}) = -Y^{-1}\Delta_{ij}\derz{}.
\]
In other words
\[
\frac{\partial I(Z)}{\partial Z_{ij}} = 
\begin{cases}
- \sum_{k=1}^{n}( w_{ki}\frac{\partial}{\partial z_j}\otimes
d\bar{z}_k  + w_{kj} \derz{i}\otimes
d\bar{z}_k)
& \text{for } i\neq j \\
- \sum_{k=1}^n w_{ki} \derz{i} \otimes d\bar{z}_k & \text{for } i = j
\end{cases}
\]
Remark that since $I(Z)^2 = -Id$, $\frac{\partial I(Z)}{\partial
  Z_{ij}}$ and $I(Z)$ anti-commute. This is clearly reflected in
the above expressions for $\frac{\partial I(Z)}{\partial Z_{ij}}$. Now since $G(\dert{ij})$ is defined by
\[
-G (\dert{ij}) \cdot \frac{1}{2i} \sum_{kl=1}^n w_{kl}dz_k\wedge d\bar{z}_l
= \frac{\partial I(Z)}{\partial Z_{ij}}
\]
it is
\[
G(\dert{ij}) =
\begin{cases} 2i \derz{i} \otimes
\derz{j} +2i \derz{j} \otimes \derz{i} & \text{for } i \neq j \\
2i \derz{i} \otimes \derz{i} & \text{for } i = j.
\end{cases}
\]
With $G(\dert{ij})$ being expressed in complex coordinates, we should
mentioned that the family of K\"{a}hler structures parametrized by
$\bH$ in the way described above, actually is rigid,
i.e. $\bar{\partial}_Z (G(V)_Z) = 0$ for all vector field
$V$ on $\bH$. This is clear since $G(\dert{ij})$ is zero in
$\bar{z}_i$ directions and $G(\derbt{ij}) = 0$.

We could do exactly the same thing with $\derbt{ij}$ and obtain
\[
\frac{\partial I(Z)}{\partial \bar{Z}_{ij}} =
-\frac{1}{2i}Y^{-1}\Delta_{ij}Y^{-1}
\begin{pmatrix}
  Z & Z^2 \\ -1 & -Z
\end{pmatrix}.
\]
Again it is clear that
\[
\frac{\partial I(Z)}{\partial \bar{Z}_{ij}}(\derz{}) = 0 \quad
\text{and} \quad \frac{\partial I(Z)}{\partial \bar{Z}_{ij}} =
-Y^{-1}\Delta_{ij} \derbz{}.
\]
In a similar way as above we obtain
\[
\bar{G}(\derbt{ij}) =
\begin{cases} 2i \derbz{i} \otimes
\derbz{j} +2i \derbz{j} \otimes \derbz{i} & \text{for } i \neq j \\
2i \derbz{i} \otimes \derbz{i} & \text{for } i = j. 
\end{cases}\qedhere
\]
\end{proof}

\nocite{*} \bibliographystyle{OUPnamed} \bibliography{biblist}

\end{document}